\documentclass[a4paper,12pt]{article}

\usepackage[utf8]{inputenc}
\usepackage{amsmath,amsfonts}
\usepackage{amsthm} 
\usepackage{amssymb}
\usepackage{multirow}
\usepackage{latexsym}
\usepackage{xcolor}
\usepackage{graphicx} 
\usepackage{caption}
\usepackage{subcaption}
\usepackage{natbib}
\usepackage[colorlinks=true,linkcolor=blue,urlcolor=blue,citecolor=blue]{hyperref}
\usepackage{bbm}
\usepackage{hyperref}
\usepackage{booktabs}
\usepackage{multirow}
\usepackage{authblk}

\usepackage[margin=1in]{geometry}

\newtheorem{theorem}{Theorem}[section] 
\newtheorem{lemma}[theorem]{Lemma}
\newtheorem{definition}[theorem]{Definition}

\newtheorem{corollary}[theorem]{Corollary}
\newtheorem{remark}[theorem]{Remark}

\newcommand{\argmax}[1]{\underset{#1}{\operatorname{argmax}}}
\newcommand{\IR}{\mathbb{R}}
\newcommand{\IE}{\mathbb{E}}
\newcommand{\IP}{\mathbb{P}}
\newcommand{\Ind}{\mathbbm{1}}
\newcommand{\IN}{\mathbb{N}}

\newcommand{\IC}{\mathbbm{C}}

\renewcommand{\tilde}{\widetilde}
\renewcommand{\epsilon}{\varepsilon}
\renewcommand{\phi}{\varphi}

\allowdisplaybreaks

\numberwithin{equation}{section}

\begin{document}
\title{\textbf{Semi-Parametric Estimation of Incubation and Generation Times by Means of Laguerre Polynomials}}
\author{Alexander Kreiss\footnote{This work was supported by the European Research Council under Grant 2016-2021, Horizon 2020 / ERC grant agreement No. 694409.}}
\affil{London School of Economics, Department of Statistics, London, UK}
\newcommand\CoAuthorMark{\footnotemark[\arabic{footnote}]}
\author[2]{Ingrid Van Keilegom\protect\CoAuthorMark}
\affil{KU Leuven, Faculty of Economics and Business (FEB), Leuven, Belgium}
\maketitle
\begin{abstract}
In epidemics many interesting quantities, like the reproduction number, depend on the incubation period (time from infection to symptom onset) and/or the generation time (time until a new person is infected from another infected person). Therefore, estimation of the distribution of these two quantities is of distinct interest. However, this is a challenging problem since it is normally not possible to obtain precise observations of these two variables. Instead, in the beginning of a pandemic, it is possible to observe for transmission pairs the time of symptom onset for both people as well as a window for infection of the first person (e.g. because of travel to a risk area). In this paper we suggest a simple semi-parametric sieve-estimation method based on Laguerre-Polynomials for estimation of these distributions. We provide detailed theory for consistency and illustrate the finite sample performance for small datasets via a simulation study.
\end{abstract}

\textbf{Keywords:} Laguerre-Polynomials, Semi-Parametric Estimation, Sieve-Estimation, Epidemics

\section{Introduction}
\label{sec:introduction}
A dominating question in the public evaluation of the COVID-19 situation during the 2020 pandemic is the estimation of the basic reproduction number $R_0$, the number of new infections which are caused (on average) from a single infected individual. While this number is prominently discussed in the news about COVID-19, it is an important variable for disease transmission in general (\citet{WL07,LLSHJW20}). In order to estimate the reproduction number of a disease the so-called generation time $G$ plays an important role (see Euler-Lotka-Equation e.g. in \citet{BS19}). The generation time $G$ is defined as the time difference between the infection time of a randomly chosen infected individual and his or her infector. Let $\phi_G$ denote the density of the generation time $G$ and let $i(t)$ be the expected incidence at time $t$ assuming a certain model of transmission (``\emph{the average community rate of new infections}'', \citet[p. 2]{BS19}). By formulating a renewal equation for $i(t)$ and assuming an exponential growth $i(t)=Ce^{rt}$, one obtains $R_0^{-1}=F_{R_0}(G)$ (the Euler-Lotka-Equation, see e.g. \citet{BS19,WL07}), where 
\begin{equation}
\label{eq:defF}
F_{R_0}(G):=\IE\left(e^{-rG}\right)=\int_0^{\infty}e^{-rt}\phi_G(t)dt.
\end{equation}
Hence, the basic reproduction number $R_0$ is a function of $\phi_G$.

We suppose to observe transmission pairs, i.e, two infected people $A$ and $B$ where it is known that $A$, the \emph{index case}, infected $B$, the \emph{secondary case}. We observe for $A$ and $B$ their times of symptom onset and we observe, in addition, that the infection of the index case happened in a certain interval. We do not assume information about the infection of the secondary case through the index case. Such data can arise e.g. if the index case got infected during a travel to a region where the virus is circulating and infected the secondary case back home where the virus is not spreading yet (see \citet{LGB20,Bietal20,XLL20} for examples related to travel to and from Wuhan in the early days of the pandemic). This leads to observations of the infection times with measurement error. Hence, the generation time and the incubation times $I_1,I_2$ (the time from infection to symptom onset) are not directly observed. The serial interval $S$ (the time between the two symptom onsets), in contrast, can be observed. Coming back to the example of estimation of $F_{R_0}(G)$ (and hence $R_0$), the most natural estimator, i.e., replacing the expectation by the sample average, is not feasible from our observations because the exact infection times are not observed. In reality it is rarely possible to make observations of $G$ directly for the same reason (cf. \citet{NKI09}). However, the urgency of the situation requires estimation with imperfect data (\citet{FWK20}).

We can write down a likelihood for observing the symptom onset times and the exposure interval. This likelihood can be written in terms of $\phi_G$ and other quantities. If it was possible to formulate a family of generation time densities which is indexed by $R_0$ we could simply maximize the likelihood over $R_0$. In view of \eqref{eq:defF} this seems difficult. We follow therefore a different idea: Find a non-parametric estimate $\hat{\phi}_{G,n}$ of $\phi_G$ and study $F_{R_0}(\hat{\phi}_{G,n})^{-1}$ as an estimator of $R_0$. Such \emph{plug-in estimators} have been studied e.g. in \citet{S97} and it was argued there that smoothness of $F_{R_0}$ can compensate for typical drawbacks of non-parametric estimators like a slow convergence rate.

This approach is not limited to $R_0$ but can be applied to other interesting functions $F$ of $\phi_G$, we call them \emph{features of $G$}. One example are tests: Some methods for estimation of $R_0$ make implicit assumptions about the generation time (cf. \citet{WL07}). Testing for such assumptions can be realized within this framework if the test statistic $T(\hat{\phi}_{G,n})$ can be written as a function of an estimate $\hat{\phi}_{G,n}$ of $\phi_G$, e.g. similarly to \citet{HM93} who use an $L^2$ type test-statistic. Further interesting features of $\phi_G$ are for instance variances, mean values, quantiles or the probability of pre-symptomatic infection $\IP(G\leq I)$ where $I$ denotes the incubation time (the time from infection to symptom onset). This is a feature of the joint variable $(G,I)$. It is therefore interesting to study efficient estimation of general features of $G$ or of $(G,I)$. We will focus in the following on features of $G$. More precisely, we will estimate $F(G)$ by firstly finding a semi-parametric sieve-estimator $\hat{\phi}_{G,n}$ of $\phi_G$ based on Laguerre polynomials and consider then the \emph{plug-in} estimator which estimates $F(G)$ through replacing $\phi_G$ by $\hat{\phi}_{G,n}$ in the definition of $F$.

In the beginning of a pandemic it is important to estimate its transmission characteristics (\citet{Bietal20}) and it is common to replace the unobserved generation time $G$ by the serial interval $S$ (cf. \citet{BS19,FWK20,WL07}), because this can be observed in clinical studies. In certain situations this can be a promising approach, cf. Remark \ref{rem:serial_vs_generation} below. However, this practice can also yield biased estimates (cf. \citet{BS19}). In this paper, we avoid this issue by estimating $\phi_G$ directly. Moreover, since the interest lies in a function of $\phi_G$ rather than $\phi_G$ itself, it might be desirable for researchers to not impose parametric assumptions on $\phi_G$ but to consider a non-parametric approach instead. In addition, a non-parametric approach avoids issues like selecting a parametric family or important aspects being hidden through this choice (cf. \citet{G21}). 

The main contribution of this paper is to provide the first step towards efficient, semi-parametric estimation of features of the generation time: Consistent density estimation. Based on such an estimator the theory for efficient, semi-parametric inference methods about features of $\phi_G$ can be developed. To this end we extend the methodology of \citet{FWK20,Getal20} in two ways: Firstly, we estimate the densities of the incubation period and the generation time jointly in one step with the same data and, secondly, we provide a semi-parametric framework which makes no a-priori assumptions about the incubation and generation times. We pursue this by constructing a sieve estimator based on Laguerre polynomials highlighting the flexibility of Laguerre polynomials as approximating functions. In order to identify the model we need to specify a model for the relation between the exposure window and the infection time but we will not make parametric assumptions about the distributions of the incubation period and the generation time. Thus, we will do semi-parametric estimation. 

General background information about data analysis in disease transmission can e.g. be found in \citet{HHOW19,CHBCC09}. The estimation of $R_0$ from observations of the serial interval in particular is discussed in \citet{LCC03}. As an alternative \citet{FWK20} suggest the following two step procedure: Firstly, fit a parametric model for the incubation time (cf. \citet{LGB20,Bietal20}) and then, secondly, use this to fit a parametric model to the generation time. \citet{G21} describes how to do non-parametric estimation of the incubation time. Problems related to under-reporting or delays (cf. \citet{AFH14}) are not considered here because we have data in mind which was collected by researchers rather than observational data from self-reporting. In addition to this type of experimental data, it is also possible to collect larger sets with covariates and use prediction techniques to identify transmission pairs (cf. \citet{LLSHJW20}). Specific parametric results about the above mentioned quantities for Covid-19 can e.g. be found in \citet{Bietal20,FWK20,Getal20,LGB20,Tetal20} and in many other places. The exact mathematical setting described below is close to de-convolution and measurement error settings, classical results about which can be found for example in \citet{F91,D89,CRSC06}. General background about sieve estimators and series estimation can be found in \citet{N97,C07} and the references therein. \citet{AC03,NP03} use sieve estimation based on moment conditions while we start from a likelihood. Further reading about plug-in sieve estimates can be found in \citet{S97,CS98}.

The structure of this paper is as follows: In Section \ref{sec:model} we introduce the exact modelling framework. Afterwards, in Section \ref{sec:methodology}, we provide a sieve estimator for the incubation period and generation time based on Laguerre polynomials. Its properties will be discussed in Section \ref{sec:theory} and we show that our estimator is asymptotically consistent. The asymptotic distribution of the methodology will be studied in Section \ref{sec:empirical_example} by means of simulations and it will be applied to a real-world dataset consisting of 191 SARS-COV-2 transmission pairs that has also been used by \citet{HMT21} and \citet{FLWZ20}. The R-code which is used for these computations is available on github (\href{https://github.com/akreiss/SemiParametric-Laguerre.git}{https://github.com/akreiss/ SemiParametric-Laguerre.git}). Finally, we finish the paper with some concluding remarks in Section \ref{sec:conclusion}. Additional simulation results and some proofs are collected in the Appendix.

\section{Model}
\label{sec:model}

We study observations of transmission pairs as shown in Figure \ref{fig:infection_chain}. The first person, the index case, gets infected at time $T_1$ which is unobserved. However, it is known to lie in the interval $[0,W]$, the \emph{exposure period}. The exact conditional distribution of $T_1$ in $[0,W]$ is allowed to depend on a random variable $C$ which determines for example the location of the infection. In certain situations, it will turn out that the likelihood will depend only on $\min(W,S_1)$, where $S_1\geq T_1$ denotes the time at which Person 1 shows symptoms. In that sense, observation of $W$ is only then required if $W\leq S_1$. Otherwise we do not have to observe $W$ but we have to observe that $W\geq S_1$. The \emph{incubation period} is defined as $I_1:=S_1-T_1$. Moreover, Person 1 is known to have infected Person 2, the secondary case. Similarly as above, we define $S_2,T_2,I_2$ for the second person: Person 2 shows symptoms at the observed time $S_2$. But, of course, the infection time $T_2$ and hence the incubation time $I_2:=S_2-T_2$ are unobserved. The generation time is thus $G:=T_2-T_1$ and it is also unobserved, while the serial interval $S:=S_2-S_1$ can be observed. Such data was for example collected by \citet{FWK20,Bietal20,LGB20}.

\begin{figure}
\centering
\includegraphics[width=0.8\textwidth]{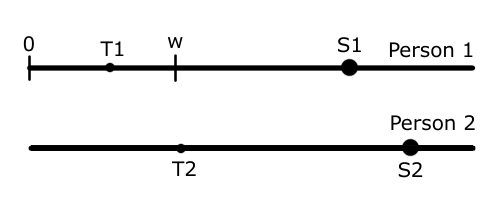}
\caption{Schematic depiction of the important quantities and their relation.}
\label{fig:infection_chain}
\end{figure}

The following relations can be directly read off from the definitions and Figure \ref{fig:infection_chain}.
\begin{align}
S_1=&T_1+I_1, \label{eq:ident1} \\
S_2=&T_1+I_2+G, \label{eq:defS2} \\
S=&G+I_2-I_1. \label{eq:ident2}
\end{align}

Let $(W_i,T_{1,i},C_i,S_{1,i},S_{2,i},G_i)_{i=1,...,n}$ be iid copies of $(W,T_1,C,S_1,S_2,G)$. However, we suppose that our observations are only given by $(S_{1,i},S_{2,i},\tilde{W}_i,C_i)_{i=1,...,n}$, where $\tilde{W}_i:=\min(S_{1,i},W_i)$. Based on our observations, we can only make direct inference about the joint distribution of $(S_1,S_2)$ and the serial interval $S:=S_2-S_1$. But we have only indirect information about $I$, $G$ and $T_1$ by means of \eqref{eq:ident1}-\eqref{eq:defS2} (note that \eqref{eq:ident2} is the difference of \eqref{eq:ident1} and \eqref{eq:defS2}). Informally speaking, we have three unknown densities and two equations. In other words, identification of the joint-distribution of the three random variables $(I,G,T_1)$ from a joint distribution of the two variables $(S_1,S_2)$ is impossible without additional assumptions. Throughout we will assume the following to be true:

\textbf{Assumption (M):} Model \\
\emph{$I_1,I_2$ and $G$ are non-negative and $C$ takes values in $\{1,...,K_C\}$. Their distributions have the following densities (with respect to the Lebesgue measure) or count measures and relate as follows
\begin{equation}
\label{eq:mod}
\begin{array}{lll}
W&\sim& \phi_W \\
C&\sim& p_C, \textrm{ independent of } W \\
T_1&\sim& \phi_{T_1}(\cdot|W,C)\,\textrm{ cond. on }\,(W,C) \\
I_1, I_2&\overset{iid}{\sim}&\phi_I, \textrm{independent of }(T_1,W,C)  \\
G&\sim& \phi_G, \textrm{ independent of } (I_1, I_2, T_1,W,C).
\end{array}
\end{equation}
It holds that $\phi_{T_1}(t|w,c)=\Ind(t\in[0,w])n(w,c)h(w-t|c)$ where $h:[0,\infty)\to[0,\infty)$ is integrable.}

When inspecting the likelihood of \citet{FWK20} and the assumptions of \citet{Getal20}, we see the same independence assumptions for $I_1, I_2, G$ and $T_1$. Independence between $I_1$ and $T_1$ can also be found in \citet{G21}. The variable $C$ can be understood as covariate and will be interpreted as the location of the infection. It can be used to capture heterogeneity in the infection dynamics: In a location with an increasing number of infections, it is more likely that the infection time $T_1$ lies towards the end of the infection window $[0,W]$ conditionally on $(C,W)$. In contrast, in locations with a low number of infections, $T_1$ is possibly uniformly distributed in $[0,W]$. See Remark \ref{rem:location} below for further details and a motivation for certain choices of $h$. It has been reported (see e.g. \citet{Tetal20}) that the transmission dynamics, i.e., $\phi_I$ and $\phi_G$, can be different in different scenarios. Thus, we need to restrict to data coming from the same environment. Moreover, the above assumption means that we observe true transmission pairs (see e.g. \citet{Tetal20} for a parametric model which allows for an unknown number of intermediate infectors).

\begin{remark}
\label{rem:location}
All discussion below is conditional on $W$. The dataset we imagine is recorded in the beginning of a pandemic, e.g. a person comes home after travelling to a region where the virus is circulating. Suppose that the time points, at which the person of interest has an infectious contact (i.e. a contact which definitely leads to an infection), are the jump points of a counting process $N$ with intensity function $\lambda:[0,W]\to[0,\infty)$. It is plausible that the intensity of infection is proportional to the number of infections in the population. Since (in most cases) the infection numbers behave exponentially, we also consider $\lambda(r)=\lambda_0\exp(-\alpha(W-r))\sim \exp(\alpha r)$, where $\alpha\in\IR$ describes the exponential growth (or decay) of the number of infections. In this notation, $T_1$, the time of infection, is the first jump of $N$. In an early stage of the pandemic it is plausible that there is only one infectious contact. Then, using that counting processes have independent increments and that the number of jumps in a certain interval is Poisson distributed with parameter equal to the integral of the intensity over the respective interval, we obtain
\begin{align*}
\IP(T_1\leq s|N(t)=1)=\frac{\IP(N([0,s])=1,\,N((s,t])=0)}{\IP(N([0,t])=1)}=\frac{\int_0^s\lambda(r)dr}{\int_0^t\lambda(r)dr}=\frac{e^{-\alpha (W-s)}-e^{-\alpha W}}{e^{-\alpha (W-t)}-e^{-\alpha W}}.
\end{align*}
Deriving the distribution function above yields that the density of $T_1$ is proportional to $\exp(-\alpha (W-s))$. Thus, we will later consider primarily the situation in which $h(u|c)=\exp(-r(C)u)$, where $r(C)$ is the (known) exponential growth of the infection numbers in location $C$. Note that taking $\alpha\to0$, i.e., the case of a constant intensity function, yields $\IP(T_1\leq s|N(t)=1)\to s/t$ (the distribution function of a uniform distribution on $[0,t]$).
\end{remark}
The discussion in Remark \ref{rem:location} is then plausible, if it is plausible to assume that the person of interest has infectious contacts with a rate which is so low/the time $W$ is so short that there is only one infectious contact in $[0,W]$. In particular in the beginning of a pandemic this seems reasonable. Note that for the subsequent infection (within the observed transmission pair) this assumption is no longer valid: It is, intuitively speaking, more likely that infections, e.g., in a household happen directly after the infected person has returned, rather than days after.

Note finally that in the special case $r(c_0)=0$ for some location $c_0$, i.e., if the infection level remains constant, we get $h(\cdot|c_0)=1$ (then $n(w,c_0)=1/w$), i.e. $T_1\sim\mathcal{U}([0,W])$ (uniform distribution). From a mathematical standpoint, this case leads to another very intuitive result: The conditional distribution of $S_1$ given $W$ and $C=c_0$ is given by the convolution through \eqref{eq:ident1}
\begin{align*}
&\IP(S_1\leq x|W,C=c_0)=\IP(T_1+I_1\leq x|W,C=c_0) \\
=&\frac{1}{W}\int_{-\infty}^{\infty}\int_{-\infty}^{\infty}\Ind(t\in[0,W])\Ind(t+s\leq x)\phi_I(s)dsdt=\frac{1}{W}\int_0^W\int_{-\infty}^{x-t}\phi_I(s)dsdt.
\end{align*}
The conditional density of $S_1$ can then be obtained through differentiation of the above with respect to $x$ and is given by
$$\frac{1}{W}\int_0^W\phi_I(x-t)dt=\frac{1}{W}\int_{x-W}^x\phi_I(u)du.$$
Hence, the likelihood of observing $S_1=x$ is given by the probability of observing an incubation period between $x-W$ and $x$, the natural lower and upper bounds. This way we can identify $\phi_I$ from observations of $S_1$. Then, in a second step, we can use \eqref{eq:ident2} to identify $\phi_G$ from observing $S$. Hence, under the assumption that $T_1\sim\mathcal{U}([0,W])$ conditionally on $W$, we can identify the distributions of $I$ and $G$ from the distribution of $(S_1,S_2)$. We make this mathematically precise also in the general case in Corollary \ref{cor:consistency} below.

We mentioned in the introduction that using the serial interval in place of the generation time can lead to biases. We discuss in the next remark in which situations this practice is safe to use and we argue that it does not work for $R_0$.

\begin{remark}[Serial Interval vs. Generation Time]
\label{rem:serial_vs_generation}
Consider the set-up from Assumption (M) and let $\phi_S$ denote the density of $S$. Our interest lies in estimation of $F(G)$, where $F$ is an arbitrary feature of $G$. As a first example let $F$ be linear, e.g. $F(G):=\IE(G)=\int_{\IR} t\phi_G(t)dt$ be the expectation. In that case by \eqref{eq:ident2}
$$F(G)=F(S)+F(I_1)-F(I_2)=F(S)=\int_{-\infty}^{\infty} t\phi_S(t)dt$$
and $F(G)$ can probably be well (maybe even optimally) estimated by estimating the integral on the right based on observations of the serial interval alone.
For the case of the basic reproduction number, however, $F_{R_0}(G)=\IE(\exp(-rG))$ is a moment generating function, i.e., it is non-linear. Other examples of interesting non-linear features of interest are quantiles, a quadratic test statistic or the probability of pre-symptomatic infection $\IP(G\leq I_1)$. For these features, we cannot simply estimate the serial interval in place of the generation time, but it is necessary to have an estimate of the density $\phi_G$ available. 
\end{remark}

Remark \ref{rem:serial_vs_generation} together with the unavailability of observations from $G$ shows that estimation of features like $R_0$ is non-trivial and requires more complicated methods. Before turning to the estimators we make a remark about asymptomatic patients.
\begin{remark}[Asymptomatic Patients] It is unclear how asymptomatic patients can be handled in this framework. However, it should be emphasized that asymptomatic means here infected people who infect others but who never show symptoms. Pre-symptomatic infections are well allowed, i.e., Person 1 is allowed to infect Person 2 before showing symptoms. However, we assume here that Person 1 and Person 2 both will eventually show symptoms.
\end{remark}

\section{Methodology}
\label{sec:methodology}
We introduce now our estimator by specifying approximations to $\phi_I$ and $\phi_G$. Consider the following class of densities on the non-negative real line for $m\in\IN_0:=\{0,1,2,...\}$
$$\mathcal{G}_m:=\left\{\phi_{\theta} :\quad \theta\in\IR^{m+1}, \|\theta\|_2=1\right\},$$
where $\|\cdot\|_2$ is the Euclidean norm and
\begin{equation}
\label{eq:def_phi_theta}
\phi_{\theta}(x)=\Ind(x\geq0)e^{-x}\left(\sum_{k=0}^m\theta^{(k)}L_k(x)\right)^2
\end{equation}
with $\theta=(\theta^{(0)},...,\theta^{(m)})$
$$L_k(x):=\sum_{i=0}^k\binom{k}{i}\frac{(-x)^i}{i!}$$
being the $k$-th Laguerre polynomial. Since the Laguerre polynomials form an orthonormal system of functions on $[0,\infty)$ with respect to the weight function $e^{-x}$ it is simple to show that the above construction yields a density under the simple condition $\|\theta\|_2=1$, cf. Lemma \ref{lem:approx}. Moreover, it is well known that a large class of densities can be approximated by Laguerre polynomials. Define to this end the Hellinger distance of two distributions on $\mathcal{X}$ which have densities $\phi_1,\phi_2$ with respect to a measure $\mu$ via
$$\rho_H(\phi_1,\phi_2):=\left(\int_{\mathcal{X}}\left(\phi_1(x)^{\frac{1}{2}}-\phi_2(x)^{\frac{1}{2}}\right)^2d\mu(x)\right)^{\frac{1}{2}}.$$
The following lemma is essentially a rephrasing of Theorem 1 in Chapter II.8 of \citet{NU88} which we state here for the convenience of the reader.
\begin{lemma}
\label{lem:approx}
For any $m\in\IN_0$ and any $\theta\in\IR^{m+1}$ with $\|\theta\|_2=1$ we have that $\phi_{\theta}$ as defined in \eqref{eq:def_phi_theta} is a density function on $[0,\infty)$. Suppose moreover that $\phi$ is an arbitrary density on $[0,\infty)$ such that $p(x):=\sqrt{e^x\phi(x)}$ is continuous on $[0,\infty)$ and has a piecewise continuous derivative $p'(x)$. Suppose that
\begin{equation}
\label{eq:cond2}
\int_0^{\infty}p'(x)^2xe^{-x}dx<+\infty.
\end{equation}
Then, for any sequence $m_n\to\infty$ there are $\theta_n\in\IR^{m_n+1}$ with $\|\theta_n\|=1$ such that for $n\to\infty$
$$\rho_H(\phi_{\theta_n},\phi)^2=\int_0^{\infty}\left(\phi_{\theta_n}(x)^{1/2}-\phi(x)^{1/2}\right)^2dx\to0.$$
Moreover, $\phi_{\theta_n}\to\phi$ locally uniformly, i.e., $\sup_{x\in K}|\phi_{\theta_n}(x)-\phi(x)|\to0$ for any compact set $K\subseteq(0,\infty)$.
\end{lemma}
\begin{proof}
Let $p_{m_n}(x;\theta_n):=\sum_{k=0}^{m_n}\theta_n^{(k)}L_k(x)$, where $\theta_n^{(k)}$ denotes the $k$-th entry of $\theta_n$. Then $\phi_{\theta_n}(x)=\Ind(x\geq0)e^{-x}p_{m_n}(x;\theta_n)^2$. The first part of the lemma is a simple calculation using the orthonormality of the Laguerre polynomials, i.e., use that
$$\int_0^{\infty}e^{-x}L_k(x)L_l(x)dx=\Ind(k=l).$$
For the second statement we note
$$\int_0^{\infty}e^{-x}p(x)^2dx=\int_0^{\infty}\phi(x)dx=1<\infty.$$
This and \eqref{eq:cond2} are exactly the conditions of Theorem 1 in Chapter II.8 of \citet{NU88} which we stated in the Appendix as Theorem \ref{thm:lag_approx} for the convenience of the reader. The theorem states that the sequence $\tilde{p}_n(x):=\sum_{k=0}^{m_n}c_kL_k(x)$ with $c_k=\int_0^{\infty}p(x)L_k(x)e^{-x}dx$ converges locally uniformly, i.e., uniformly on compact sets, to $p$. Let
$$\mathcal{C}_n:=\int_0^{\infty}e^{-x}\tilde{p}_n(x)^2dx=\sum_{k=0}^{m_n}c_k^2$$
and $\theta_n^{(k)}:=c_k/\mathcal{C}_n^{1/2}$ for $k=0,...,m_n$. Clearly, $\|\theta_n\|_2=1$. By Parseval's Identity $\mathcal{C}_n\to\sum_{k=0}^{\infty}c_k^2=\int_0^{\infty}e^{-x}p(x)^2dx=1$. Hence, we can conclude that also
$$p_{m_n}(x;\theta_n)=\sum_{k=0}^{m_n}\theta_n^{(k)}L_k(x)=\mathcal{C}_n^{-\frac{1}{2}}\sum_{k=0}^{m_n}c_kL_k(x)=\mathcal{C}_n^{-\frac{1}{2}}\tilde{p}_n(x)$$
converges locally uniformly to $p$ because $p$ is bounded on compact sets $K\subseteq(0,\infty)$. From this we conclude also the locally uniform convergence of $\phi_{\theta_n}$ to $\phi$. Moreover, Theorem \ref{thm:lag_approx} states that 
$$\int_0^{\infty}e^{-x}(\tilde{p}_n(x)-p(x))^2dx\to0.$$
The proof of the Lemma is complete since the above implies
\begin{align*}
&\int_0^{\infty}\left(\phi_{\theta_n}^{1/2}-\phi^{1/2}\right)^2dx=\int_0^{\infty}e^{-x}\left(p(x)-p_{m_n}(x;\theta_n)\right)^2dx \\
=&\int_0^{\infty}e^{-x}\left(p(x)-\tilde{p}_n(x)+\tilde{p}_n(x)\left(1-\mathcal{C}_n^{-\frac{1}{2}}\right)\right)^2dx \\
\leq&2\int_0^{\infty}e^{-x}\left(p(x)-\tilde{p}_n(x)\right)^2dx+2\int_0^{\infty}e^{-x}\tilde{p}_n(x)^2dx\left(1-\mathcal{C}_n^{-\frac{1}{2}}\right)^2\to0.
\end{align*}
\end{proof}
The conditions of Lemma \ref{lem:approx} cover a wide class of piecewise continuously differentiable densities like any sub-Gaussian density, densities with compact support and sub-exponential densities. We consider condition \eqref{eq:cond2} therefore as not-restrictive. Before we can write down the likelihood, we have to find the conditional density of $(S_1,S_2,W)$ given $C$. 
\begin{lemma}
\label{lem:density}
Suppose that the random variables $W,C,T_1,I_1,I_2,G$ are related as in Assumption (M). The conditional joint density of $(S_1,S_2,W)$ given $C$, where $S_1:=I_1+T_1$ and $S_2:=T_1+I_2+G$, is given by
$$f(x_1,x_2,\omega|c)=n(\omega,c)\phi_W(\omega)\int_0^{x_2}\phi_G(y)\int_0^{\min(x_1,\omega)}h(\omega-t|c)\phi_I(x_1-t)\phi_I(x_2-t-y)dtdy.
$$
for $x_1,x_2,\omega\geq0$ and $0$ otherwise.
\end{lemma}
\begin{proof}
Let $x_1,x_2,\omega\geq0$ be arbitrary. Note firstly that, by (M), the conditional density of $(W,T_1)$ given $C$ is given by
$$\phi_{T_1}(t|w,c)\phi_W(w)=\Ind(t\in[0,\omega])n(w,c)h(w-t|c)\phi_W(w).$$
Below integrals of the type $\int_0^a$ are to be understood over the set $(\min(0,a),\max(0,a))$. Since all integrands are supported in $[0,\infty)$, integrals are zero when $a<0$. We have by the independence assumptions in (M)
\begin{align*}
&\IP(S_1\leq x_1, S_2\leq x_2,W\leq\omega|C)=\IP(T_1+I_1\leq x_1,\, T_1+G+I_2\leq x_2,W\leq\omega|C) \\
=&\IE\left(\Ind(W\leq\omega)\IP(T_1+I_1\leq x_1,\, T_1+G+I_2\leq x_2|T_1,W,C)|C\right) \\
=&\IE\left(\Ind(W\leq\omega)\IP(I_1\leq x_1-T_1|T_1,C)\cdot\IP(G+I_2\leq x_2-T_1|T_1,C)|C\right) \\
=&\IE\left(\Ind(W\leq\omega)\int_0^{x_1-T_1}\phi_I(a)da\cdot\int_0^{x_2-T_1}\int_0^{\infty}\phi_G(y)\phi_I(b-y)dydb\Big|C\right) \\
=&\int_0^{\omega}\int_0^vn(v,c)h(v-t|c)\phi_W(v)\int_0^{x_1-t}\phi_I(a)da\cdot\int_0^{x_2-t}\int_0^{\infty}\phi_G(y)\phi_I(b-y)dydbdtdv.
\end{align*}
We obtain furthermore by differentiating under the integral and using that $\phi_I$ and $\phi_G$ are supported on the non-negative real line 
\begin{align*}
&f(x_1,x_2,\omega|c)=\frac{d^3}{dx_1dx_2d\omega}\IP(S_1\leq x_1,S_2\leq x_2,W\leq\omega|C=c) \\
=&\frac{d^2}{dx_1dx_2}n(\omega,c)\phi_W(\omega)\int_0^{\omega}h(\omega-t|c)\int_0^{x_1-t}\phi_I(a)da\cdot\int_0^{x_2-t}\int_0^{\infty}\phi_G(y)\phi_I(b-y)dydbdt \\
=&n(\omega,c)\phi_W(\omega)\int_0^{\omega}h(\omega-t|c)\phi_I(x_1-t)\cdot\int_0^{\infty}\phi_G(y)\phi_I(x_2-t-y)dydt \\
=&n(\omega,c)\phi_W(\omega)\int_0^{\infty}\phi_G(y)\int_0^{\omega}h(\omega-t|c)\phi_I(x_1-t)\phi_I(x_2-t-y)dtdy \\
=&n(\omega,c)\phi_W(\omega)\int_0^{x_2}\phi_G(y)\int_0^{\min(x_1,\omega)}h(\omega-t|c)\phi_I(x_1-t)\phi_I(x_2-t-y)dtdy.
\end{align*}
In the last line above we can replace the upper integration bound of the outer integral by $x_2$ because for $y>x_2$, we have $x_2-t-y<0$ and hence $\phi_I(x_2-t-y)=0$.
\end{proof}

By using the density from above, we obtain an expression for the likelihood for estimation of $\phi_I$ and $\phi_G$.

\begin{lemma}
\label{lem:likelihood}
Let the observations $(S_{1,i},S_{2,i},W_i,C_i)_{i=1,...,n}$ be iid copies of $(S_1,S_2,W,C)$ which are related as in Assumption (M). Let $m_1,m_2\in\IN$ and $\theta_1\in\IR^{m_1+1}$, $\theta_2\in\IR^{m_2+1}$. Consider the candidate densities $\phi_{I,\theta_1}\in\mathcal{G}_{m_1}$ and $\phi_{G,\theta_2}\in\mathcal{G}_{m_2}$. The likelihood for this configuration is given by, denote $\tilde{W}_i:=\min(W_i,S_{1,i})$
\begin{align}
&\mathcal{L}_n(\phi_{I,\theta_1},\phi_{G,\theta_2}) \nonumber \\
=&\sum_{i=1}^n\log\left(\int_0^{S_{2,i}}\phi_{G,\theta_2}(y)\int_0^{\tilde{W}_i}h(W_i-t|C_i)\phi_{I,\theta_1}(S_{1,i}-t)\phi_{I,\theta_1}(S_{2,i}-t-y)dtdy\right) \nonumber \\
&\quad\quad\quad+C_n, \label{eq:loglik_basis}
\end{align}
where $C_n$ is a random constant which does not depend on $\phi_{I,\theta_1}$ and $\phi_{G,\theta_2}$. In the specific case of $h(u|c)=\exp(-r(c)u)$ for some exponential parameter $r(c)\in\IR$ we even get that
\begin{align}
\mathcal{L}_n=&\sum_{i=1}^n\log\left(\int_0^{S_{2,i}}\phi_{G,\theta_2}(y)\int_0^{\tilde{W}_i}\exp\left(r(C_i)t\right)\phi_{I,\theta_1}(S_{1,i}-t)\phi_{I,\theta_1}(S_{2,i}-t-y)dtdy\right) \nonumber \\
&\quad\quad\quad\quad+\tilde{C}_n \label{eq:loglik_basis_exp}
\end{align}
depends on $W_i$ only through $\tilde{W}_i$. Above $\tilde{C}_n$ is a random constant which does not depend on $\phi_{I,\theta_1}$ and $\phi_{G,\theta_2}$.
\end{lemma}
\begin{proof}
Let $f_{\phi_{I,\theta_1},\phi_{G,\theta_2}}(x_1,x_2,\omega|c)$ denote the conditional density of $(S_1,S_2,W)$ given $C$ as defined in Lemma \ref{lem:density} when $I$ has density $\phi_{I,\theta_1}$ and $G$ has density $\phi_{G,\theta_2}$. By independence, we have
\begin{align*}
&\mathcal{L}_n \\
=&\sum_{i=1}^n\log \left(f_{\phi_{I,\theta_1},\phi_{G,\theta_2}}(S_{1,i},S_{2,i},W_i|C_i)p_C(C_i)\right) \\
=&\sum_{i=1}^n\log\left(\int_0^{S_{2,i}}\phi_{G,\theta_2}(y)\int_0^{\min(S_{1,i},W_i)}h(W_i-t|C_i)\phi_{I,\theta_1}(S_{1,i}-t)\phi_{I,\theta_1}(S_{2,i}-t-y)dtdy\right) \\
&\quad\quad\quad\quad+\sum_{i=1}^n\log\left(n(W_i,C_i)\phi_W(W_i)p_C(C_i)\right) ,
\end{align*}
where the latter term is independent of $\phi_{I,\theta_1}$ and $\phi_{G,\theta_2}$. This finishes the proof of \eqref{eq:loglik_basis}. For the specific choice $h(u|c)=\exp(-r(c)u)$ we get
\begin{align*}
&\mathcal{L}_n \\
=&\sum_{i=1}^n\log\left(\int_0^{S_{2,i}}\phi_{G,\theta_2}(y)\int_0^{\min(S_{1,i},W_i)}\exp\left(r(C_i)t\right)\phi_{I,\theta_1}(S_{1,i}-t)\phi_{I,\theta_1}(S_{2,i}-t-y)dtdy\right) \\
&\quad\quad\quad\quad+\sum_{i=1}^n\left(\log\left(n(W_i,C_i)\phi_W(W_i)p_C(C_i)\right)-r(C_i)W_i\right)
\end{align*}
which finishes the proof of \eqref{eq:loglik_basis_exp} since the second line does not depend on $\phi_{I,\theta_1}$ or $\phi_{G,\theta_2}$.
\end{proof}
Note that Lemma \ref{lem:likelihood} implies that for given $\phi_W$ and $p_C$, the likelihood to be optimized is, conveniently, independent of $\phi_W$ and $p_C$. Thus, we can treat $\phi_W$ and $p_C$ as known without loosing any practicality and, therefore, we may define the following estimators for $\phi_I$ and $\phi_G$: For given sequences $(m_{1,n})_{n\in\IN},(m_{2,n})_{n\in\IN}\subseteq\IN_0$, we study
\begin{align*}
\left(\hat{\phi}_{I,n},\hat{\phi}_{G,n}\right):=&\argmax{\phi_{I,\theta_1}\in\mathcal{G}_{m_1},\phi_{G,\theta_2}\in\mathcal{G}_{m_2}}\sum_{i=1}^n\log\Bigg(\int_0^{S_{2,i}}\phi_{G,\theta_2}(y)\int_0^{\tilde{W}_i}h(W_i-t|C_i) \\
&\quad\quad\quad\quad\times \phi_{I,\theta_1}(S_{1,i}-t)\phi_{I,\theta_1}(S_{2,i}-t-y)dtdy\Bigg).
\end{align*}
At this point it is not clear that we can identify the parameters $\theta_1$,$\theta_2$. Later, in Corollary \ref{cor:consistency}, we will see that consistent estimation of the distribution functions of $\phi_I$ and $\phi_G$ is possible if a mild assumption on the characteristic functions holds. Under the same assumption, a similar proof-technique can be applied to show that two different sets of parameters lead to different likelihoods.

\section{Theory}
\label{sec:theory}
In this section we will prove and discuss a consistency result. For $\phi_I$ and $\phi_G$ being arbitrary densities of $I_1,I_2$ and $G$, respectively, we denote by $f_{\phi_I,\phi_G}$ the conditional density of $(S_1,S_2,W)$ given $C$ as defined in Lemma \ref{lem:density}. Then, $f_{\phi_I,\phi_G}p_C$ denotes the joint density of $(S_1,S_2,W,C)$. When $\phi_I$ and $\phi_G$ are chosen from the approximating spaces $\mathcal{G}_{m_1}$ and $\mathcal{G}_{m_2}$, the following set denotes the set of all possible approximations to the joint density of $(S_1,S_2,W,C)$
$$\mathcal{F}_{m_1,m_2}:=\left\{f_{\phi_I,\phi_G}p_C:\,\phi_I\in\mathcal{G}_{m_1},\,\phi_G\in\mathcal{G}_{m_2}\right\}.$$
In the following we will always assume that we are in the setting presented in Section \ref{sec:model}.
\begin{theorem}
\label{thm:consistency}
Suppose (M) holds true and let $m_{1,n},m_{2,n}=n^{\beta}$ and $\epsilon_n=n^{-\gamma}$ for some $\beta,\gamma>0$ such that $\gamma<1/2$ and $\beta\leq1-2\gamma$. Suppose that the true densities $\phi_I$ and $\phi_G$ fulfil the assumptions of Lemma \ref{lem:approx} and let $\phi_{I,n}$ and $\phi_{G,n}$ be the sequences from Lemma \ref{lem:approx} such that $\rho_H(\phi_I,\phi_{I,n})\to0$ and $\rho_H(\phi_G,\phi_{G,n})\to0$. Suppose that there is $\alpha\in(0,1/2)$ such that for some $N_0\in\IN_0$
\begin{equation}
\label{eq:tail_cond}
\sup_{n\geq N_0}C_{\alpha}(f_{\phi_I,\phi_G},f_{\phi_{I,n},\phi_{G,n}})<\infty,
\end{equation}
where
\begin{equation}
\label{eq:defC}
C_{\alpha}(f_1,f_2):=\left(\sum_{c=1}^{K_C}\int_{[0,\infty)^3}\left(\frac{f_1(x_1,x_2,\omega|c)}{f_2(x_1,x_2,\omega|c)}\right)^{\frac{\alpha}{1-\alpha}}f_1(x_1,x_2,\omega|c)p_C(c)d(x_1,x_2,\omega)\right)^{1-\alpha}.
\end{equation}
Then, we have that
$$\rho_H(f_{\phi_I,\phi_G}p_C,f_{\hat{\phi}_{I,n},\hat{\phi}_{G,n}}p_C)=O_P\left(\max\left(\epsilon_n,\max\left(\rho_H(\phi_I,\phi_{I,n}),\rho_H(\phi_G,\phi_{G,n})\right)^{\alpha}\right)\right)=o_P(1).$$
\end{theorem}
The proof of this result will be given later in Section \ref{subsec:proof}. We will begin with a discussion of the result and its assumptions in the next subsection.

\subsection{Discussion of Theorem \ref{thm:consistency}}
\label{subsec:discussion}
Before turning to the assumptions of Theorem \ref{thm:consistency} we make a remark about the convergence rate.
\begin{remark}
Note that the requirements on $\beta$ and $\gamma$ in Theorem \ref{thm:consistency} allow for the choice $\beta=1/5$ and $\gamma=2/5<1/2$. This yields the \emph{classical} convergence rate $n^{-2/5}$ for series estimation with polynomials if the true density is sufficiently smooth (cf. \citet{N97} and Proposition 3.6 in \citet{C07}).
\end{remark}
We continue with a discussion of the condition \eqref{eq:tail_cond} above. This can be understood as a tail condition: For any compact set $K\subseteq(0,\infty)^3$ we have by Lemma \ref{lem:approx} that $f_{\phi_{I,n},\phi_{G,n}}$ converges uniformly to $f_{\phi_I,\phi_G}$. Thus, if we restrict the integral in the definition of $C_{\alpha}(f_{\phi_I,\phi_G},f_{\phi_{I,n},\phi_{G,n}})$ to $K$, the restricted integral remains bounded. Note also that by choosing $\alpha>0$ small, possible singularities of $f_2$ are integrable. Hence, \eqref{eq:tail_cond} is only restrictive for the integral over $K^c$. Thus, if $f_{\phi_I,\phi_G}$ is compactly supported \eqref{eq:tail_cond} follows. Note furthermore that $f_{\phi_I,\phi_G}$ is compactly supported if $\phi_I, \phi_G$ and $\phi_W$ are compactly supported. In our specific epidemics setting for COVID-19, this is a highly plausible assumption because people stop being infectious at some point. However, for other diseases people can stay infectious for an extended period of time, e.g. for Hepatitis C, the mean generation time is about $20$ years (cf. \citet{WL07}).

In those cases of non-compactly supported distributions, \eqref{eq:tail_cond} is a restriction: It is required that the approximations do not decrease much faster to zero than the actual density. The meaning of \emph{much faster} is to be understood relative to $f_{\phi_I,\phi_G}$ and it can be adjusted by choosing $\alpha>0$ small. Any polynomial difference can therefore be handled.

Theorem \ref{thm:consistency} above shows, strictly speaking, that we can consistently estimate the joint distribution of $(S_1,S_2,W,C)$. However, from a practical point of view, the estimation of $\phi_I$ and $\phi_G$ is of interest. In order to ensure identifiability of $\phi_I$ and $\phi_G$, we  need the following additional assumption:

\textbf{Assumption (C):} Characteristic Functions: \\
\emph{Let $\Phi_I$ denote the characteristic function of $I$. There are functions $z_1:\IR\to\IR$ and $z_2:\IR\to\IR$ such that for almost all $x\in\IR$
$$\Phi_I(x)\neq0\quad\textrm{ and }\quad \IE\left(e^{iz_1(x)W+iz_2(x)C}\IE\left(e^{ixT_1}\big|W,C\right)\right)\neq0.$$}

In particular the second requirement above looks cryptic in full generality. In the following remark we show that in specific scenarios it simplifies to simpler conditions which are more interpretable.
\begin{remark}
The condition
\begin{equation}
\label{eq:ccf}
\IE\left(e^{iz_1(x)W+iz_2(x)C}\IE\left(e^{ixT_1}\big|W,C\right)\right)\neq0
\end{equation}
above can be rewritten in more specific scenarios.
\begin{enumerate}
\item If $\Phi_{T_1}(x):=\IE(\exp(ixT_1))\neq0$ almost everywhere, \eqref{eq:ccf} holds because we may simply choose $z_1\equiv z_2\equiv0$.
\item If $C\equiv1$ and $T_1$ is uniformly distributed in $[0,W]$, i.e., $h(w-t|c)\equiv1$ and $n(w,c)=1/w$, we obtain with $z_2\equiv0$ and $z_1(x)=-x/2$ that (note that $\sin(x)/x\to1$ for $x\to0$ such that the below inequality chain can be used also for $x=0$)
\begin{align*}
&\IE\left(e^{iz_1(x)W+iz_2(x)C}\IE\left(e^{ixT_1}\big|W,C\right)\right)=\IE\left(e^{-i\frac{x}{2}W}\frac{1}{W}\int_0^We^{ixt}dt\right) \\
=&\IE\left(\frac{1}{ixW}\left(e^{i\frac{x}{2}W}-e^{-i\frac{x}{2}W}\right)\right)=\IE\left(\frac{1}{\frac{x}{2}W}\sin\left(\frac{x}{2}W\right)\right)=\IE\left(\int_0^1\cos\left(\frac{x}{2}Wt\right)dt\right) \\
=&\int_0^1\textrm{Re}\left(\IE\left(e^{i\frac{x}{2}Wt}\right)\right)dt=\int_0^1\textrm{Re}\left(\Phi_W\left(\frac{x}{2}t\right)\right)dt,
\end{align*}
where $\Phi_W$ is the characteristic function of $W$. The above is non-zero for almost all $x\in\IR$ (implying \eqref{eq:ccf}) for many standard distributions like the exponential distribution which characteristic function has a strictly positive real part.
\end{enumerate}
\end{remark}
Our setting is very similar to the de-convolution set-up: In \eqref{eq:ident1} the signal $I_1$ is perturbed by the noise $T_1$ (the distribution of which is determined by $(W,C)$) and in \eqref{eq:defS2} the signal $G$ is perturbed by the noise $I_2+T_1$. Thus it is clear that assumptions on $W,C$ and $I_1$ similar to those for the de-convolution set-up are necessary for identification. The assumption of almost everywhere non-vanishing characteristic functions has already been mentioned in \citet{D89} (see the Remark below Theorem 1 therein) to be necessary to guarantee consistent estimation.

Let $F_I$ and $F_G$ denote the distribution functions corresponding to the densities $\phi_I$ and $\phi_G$, respectively. Similarly, define $\hat{F}_{I,n}$ and $\hat{F}_{G,n}$ to be the distribution functions of $\hat{\phi}_{I,n}$ and $\hat{\phi}_{G,n}$, respectively. The following corollary ensures under Assumption (C) that $\hat{F}_{I,n}$ and $\hat{F}_{G,n}$ are consistent estimators for $F_I$ and $F_G$, respectively.

\begin{corollary}
\label{cor:consistency}
Suppose that, in addition to the assumptions of Theorem \ref{thm:consistency}, Assumption (C) holds true. We then have
$$\left\|\hat{F}_{I,n}-F_I\right\|_{\infty}=o_P(1) \textrm{ and } \left\|\hat{F}_{G,n}-F_G\right\|_{\infty}=o_P(1).$$
\end{corollary}

The detailed proof is given in Section \ref{subsec:cor_proofs} in the Appendix.

We finally come back to estimating the basic reproduction number $R_0$. The above framework provides us with a methodology to consistently estimate the reproduction number without making parametric statements about the incubation period or the generation time. From \eqref{eq:defF} it is evident that we require next to an estimator for $\phi_G$ also an estimate $\hat{r}_n$ of the growth rate of the expected incidence. In \citet{FWK20} the exponential growth rate of the reported numbers is estimated to be $r=0.14$. But, strictly speaking, what we need here is the growth rate of the infection numbers which is (intuitively speaking) very similar but can be different due to under-reporting and delays.  The estimator is specified in the following corollary, the proof of which can also be found in Section \ref{subsec:cor_proofs} in the Appendix.

\begin{corollary}
\label{cor:R}
Let $\hat{r}_n$ denote an estimate of the exponential growth rate of the expected incidence $r>0$, i.e., let $\hat{r}_n-r=o_P(1)$. Under the conditions of Theorem \ref{thm:consistency} and Assumption (C), we have
$$\int_0^{\infty}e^{-\hat{r}_nt}\hat{\phi}_{G,n}(t)dt\overset{\IP}{\to}F_{R_0}(G).$$
\end{corollary}

\subsection{Proof of Theorem \ref{thm:consistency}}
\label{subsec:proof}
In order to prove Theorem \ref{thm:consistency} we need to introduce the same distance relation which was used also by \citet{WS95}. Define for any two densities $f_1,f_2$ with respect to a measure $\mu$ on a space $\mathcal{X}$ and $\alpha\neq0$ the distance relation
$$\rho_{\alpha}(f_1,f_2):=\frac{1}{\alpha}\int_{\mathcal{X}}f_1(x)\left(\left(\frac{f_1(x)}{f_2(x)}\right)^{\alpha}-1\right)d\mu(x).$$
Note that
\begin{align*}
&\rho_{-1/2}(f_1,f_2)=-2\int_{\mathcal{X}}\left(\sqrt{f_1(x)f_2(x)}-f_1(x)\right)d\mu(x)=2-2\int_{\mathcal{X}}\sqrt{f_1(x)f_2(x)}d\mu(x) \\
=&\int_{\mathcal{X}}\left(f_1(x)+f_2(x)-2\sqrt{f_1(x)f_2(x)}\right)d\mu(x)=\int_{\mathcal{X}}\left(\sqrt{f_1(x)}-\sqrt{f_2(x)}\right)^2d\mu(x) \\
=&\rho_H(f_1,f_2)^2.
\end{align*}
A second property of $\rho_{\alpha}$ we mention here, is non-negativity for $\alpha\geq-1$: Let $g_{\alpha}:(0,\infty)\to\IR$ be given by $g_{\alpha}(x):=\frac{1}{\alpha}\left(x^{-\alpha}-1\right)$. Then $g_{\alpha}''(x)=(\alpha+1)x^{-\alpha-2}\geq0$ because $\alpha\geq-1$. Hence, $g_{\alpha}$ is a convex function. We hence obtain by Jensen's Inequality (for the measure with density $f_1$ with respect to $\mu$)
\begin{align*}
\rho_{\alpha}(f_1,f_2)=\int_{\mathcal{X}}g_{\alpha}\left(\frac{f_2(x)}{f_1(x)}\right)f_1(x)d\mu(x)\geq g_{\alpha}\left(\int_{\mathcal{X}}\frac{f_2(x)}{f_1(x)}f_1(x)\mu(x)\right)=0.
\end{align*}
Lemma \ref{lem:approx} shows that the sieve spaces $\mathcal{G}_m$ lie dense in a large class of densities with respect to the Hellinger distance. However, for our later consistency result, we require that the sieve spaces provide also good approximations with respect to $\rho_{\alpha}$ for $\alpha>0$. The following lemma provides the main tool.

\begin{lemma}
\label{lem:rho}
Let $\phi_{I,1}, \phi_{I,2}$ be two densities for incubation periods and let $\phi_{G,1}, \phi_{G,2}$ be two densities of generation times and consider $f_1(x_1,x_2,\omega|c)$ and $f_2(x_1,x_2,\omega|c)$ defined as in Lemma \ref{lem:density} using $(\phi_{I,1},\phi_{G,1})$ and $(\phi_{I,2},\phi_{G,2})$, respectively. Define $C_{\alpha}(f_1,f_2)$ as in \eqref{eq:defC}. Then, for any $\alpha\in(0,1/2)$,
\begin{equation}
\label{eq:hellinger_bound}
\rho_{\alpha}(f_1p_C,f_2p_C)\leq\frac{1}{\alpha}C_{\alpha}(f_1,f_2)\left(2\rho_H(\phi_{G,1},\phi_{G,2})^2+8\rho_H(\phi_{I,1},\phi_{I,2})^2\right)^{\alpha}.
\end{equation}
\end{lemma}
The proof of the above Lemma is presented in Section \ref{subsec:lemma_proofs} in the Appendix. Before we can prove the consistency result, we need as a last preparation a bound on the bracketing entropy of the spaces $\mathcal{F}_{m_1,m_2}$.

\begin{definition}
\label{eq:brackets}
Let $\epsilon>0$, $\mathcal{F}_0\subseteq\mathcal{F}$ classes of functions and $\rho:\mathcal{F}\times\mathcal{F}\to[0,\infty)$ a metric on $\mathcal{F}$ be given. The \emph{bracketing number} $\mathcal{N}_{[]}(\epsilon,\mathcal{F}_0,\rho)$ is defined as the smallest number of pairs $(l_i,u_i)\in\mathcal{F}^2$ of functions such that $\rho(l_i,u_i)\leq\epsilon$ for all pairs and such that for any $f\in\mathcal{F}_0$ there is a pair $(l_i,u_i)$ with $l_i(x)\leq f(x)\leq u_i(x)$. The pairs $(l_i,u_i)$ are called \emph{brackets}.
\end{definition}

\begin{lemma}
\label{lem:bracketing}
For $m_1,m_2\geq1$ and any $\epsilon>0$ with $\epsilon\leq\left(15m_1m_2^2/4\right)^{1/2}$ we have
$$\log\mathcal{N}_{[]}(\epsilon,\mathcal{F}_{m_1,m_2},\rho_H)\leq(m_1+m_2)\log\left(\frac{\pi\sqrt{15m_1m_2^2}}{\epsilon}\right)$$
\end{lemma}

This Lemma is also proven in Section \ref{subsec:lemma_proofs} in the Appendix. We have now all ingredients together to prove the main result of this paper.

\begin{proof}[Proof of Theorem \ref{thm:consistency}]
This theorem is a consequence of Theorem 4 from \citet{WS95}. For the convenience of the reader, we have stated the result in Section \ref{subsec:WongShen} in the Appendix as Theorem \ref{thm:WS95}. Since we are interested in an asymptotic result we may assume below that $n\geq N_0$.

We apply Theorem \ref{thm:WS95} with $Y_i=(S_{1,i},S_{2,i},W_i,C_i)$ and $\mathcal{F}_n=\mathcal{F}_{m_{1,n},m_{2,n}}$. We show next that the entropy condition is fulfilled: We note firstly that $(15m_{1,n}m_{2,n}^2)^{\frac{1}{2}}\to\infty$ and $\epsilon_n\to0$. Hence, the condition $\epsilon_n\leq (15m_{1,n}m_{2,n}^2)^{1/2}$ is eventually fulfilled and we may apply Lemma \ref{lem:bracketing}. Thus, we obtain for a suitable $c_2<\infty$
\begin{align*}
&\frac{1}{\sqrt{n}\epsilon_n^2}\int_{\frac{\epsilon_n^2}{2^8}}^{\sqrt{2}\epsilon_n}\sqrt{\log\mathcal{N}_{[]}\left(\frac{u}{c_1},\mathcal{F}_{m_{1,n},m_{2,n}},\rho_H\right)}du \\
\leq&\frac{1}{\sqrt{n}\epsilon_n^2}\int_{\frac{\epsilon_n^2}{2^8}}^{\sqrt{2}\epsilon_n}\sqrt{(m_{1,n}+m_{2,n})\log\left(\frac{c_1\pi\sqrt{15m_{1,n}m_{2,n}^2}}{u}\right)}du \\
\leq&\frac{\sqrt{2}\epsilon_n}{\sqrt{n}\epsilon_n^2}\sqrt{(m_{1,n}+m_{2,n})\log\left(\frac{2^8c_1\pi\sqrt{15m_{1,n}m_{2,n}^2}}{\epsilon_n^2}\right)} \\
\leq&2n^{-\frac{1}{2}+\gamma+\frac{1}{2}\beta}\sqrt{\log\left(2^8c_1\pi\sqrt{15}n^{\frac{3}{2}\beta+2\gamma}\right)}\leq c_2,
\end{align*}
because $-\frac{1}{2}+\gamma+\frac{1}{2}\beta\leq0$. Hence, the entropy condition of Theorem \ref{thm:WS95} is fulfilled. Moreover, by the assumptions and Lemma \ref{lem:rho} we find that
$$\delta_n(\alpha)\leq\rho_{\alpha}(f_{\phi_I,\phi_G}p_C,f_{\phi_{I,n},\phi_{G,n}}p_C)\to0.$$
Hence, we eventually have $\delta_n(\alpha)<1/\alpha$. Under these conditions, we have that $\epsilon_n^*(\alpha)\to0$ and $n\epsilon_n^*(\alpha)^2\geq n\epsilon_n^2=n^{1-2\gamma}\to\infty$ since $\gamma<1/2$. Hence, we conclude from Theorem \ref{thm:WS95} that $\rho_H(f_{\phi_I,\phi_G}p_C,f_{\hat{\phi}_{I,n},\hat{\phi}_{G,n}}p_C)=O_P(\epsilon_n^*(\alpha))$ and the proof of the theorem is complete since
$$\epsilon_n^*(\alpha)=O(\max(\epsilon_n,\max(\rho_H(\phi_I,\phi_{I,n}),\rho_H(\phi_G,\phi_{G,n}))^{\alpha})$$
by Lemma \ref{lem:rho}.
\end{proof}

\section{Empirical Studies}
\label{sec:empirical_example}
The analytic evaluation of the likelihood presented in Lemma \ref{lem:likelihood} is very tedious. Therefore, the following results were obtained by numerically approximating the integrals. As illustrated by the simulations below this approximation does not cause problems in the estimation. However, we believe that a speed-up of the method is possible if the integrals are analytically computed. In order to efficiently enforce the constraint that $\|\theta_1\|_2=\|\theta_2\|_2=1$, we optimize the angles of the polar-coordinates of $\theta_1$ and $\theta_2$ and fix their radii to $1$. The angles can then be optimized under the box constraint $[0,\pi]$ (note that $\theta$ and $-\theta$ yield the same model). Finally, we note that the likelihood can have singularities because the Laguerre densities \eqref{eq:def_phi_theta} can be zero. We overcome the resulting problems in the optimisation, similar to \citet{ZD08}, by repeating the optimisation of the likelihood with several randomly chosen starting values. Below we always considered $10$ different, random starting values.

We begin with an illustration of our methodology with synthetic data in Section \ref{subsec:simulation} followed by a brief analysis of a real-world dataset in Section \ref{subsec:real_data}. The R-code which is used for both parts is available on github (\href{https://github.com/akreiss/SemiParametric-Laguerre.git}{https://github.com/akreiss/SemiParametric-Laguerre.git}).

\subsection{Simulation Study}
\label{subsec:simulation}
It appears to be very difficult to obtain data on transmission pairs, e.g. an early dataset of \cite{FWK20} contains only $40$ transmission pairs (the data is available on the website of the journal on \href{https://doi.org/10.1126/science.abb6936}{https://doi.org/10.1126/science.abb6936} in the supplementary material section). To this end, we regard it useful to illustrate in the simulation experiments that our methodology works also for small datasets and we choose $n=40$ observations as well. In this simulation study we use as a data generating process the model as in \eqref{eq:mod} with the following choices:

\begin{center}
\begin{tabular}{l|l}
$\phi_W$     & exponential with rate $\lambda=0.3820225$ \\
$p_C$        & $\IP(C=0)=0.65$ and $\IP(C=1)=0.35$ \\
$\phi_{T_1}$ & $\phi_{T_1}(t|w,c)=\Ind(t\in[0,w])\left\{\begin{array}{lll} \frac{re^{-r(w-t)}}{1-e^{-rw}} &,& c=1 \\ \frac{1}{w} &,& c=0 \end{array}\right..$ \\
$r$          & $r=\log(2)/5$ \\
$\phi_I$     & $\phi_{I,0}=$log-normal distribution: meanlog$=1.644$, standard dev.$=0.363$ \\
$\phi_G$     & $\phi_{G,0}=$Weibull distribution: shape$=2.826$, scale$=5.665$
\end{tabular}
\end{center}

All quantities above are chosen to imitate the data used in \citet{FWK20}: The average window length $\IE(W_i)$ equals roughly the average length of the observed exposure windows in the dataset of \citet{FWK20}. Likewise, in their dataset the authors distinguish two locations, those with an exponential growth of cases with growth rate $\log(2)/5$ (in roughly 35\% of the cases) and those with no exponential growth (in roughly 65\% of the cases). Moreover, the density of the incubation period $\phi_I=\phi_{I,0}$ is the distribution of incubation times fitted by \citet{FWK20,LGB20}. Lastly, the density of the generation time is taken from \citet{FWK20}. We show firstly in Figure \ref{fig:approx} that these densities can be well approximated through Laguerre type densities as defined in \eqref{eq:def_phi_theta}. The graphs show those Laguerre densities which minimize the Hellinger distance to the true densities. Not-surprisingly, larger choices of the degrees $m_1$ and $m_2$ lead to better approximations. However, we also see that, in both cases, already degree two yields reasonable approximations. Note in particular that the flat beginning of the density of the incubation period can be well captured by the approximating densities.

\begin{figure}
\includegraphics[width=\textwidth]{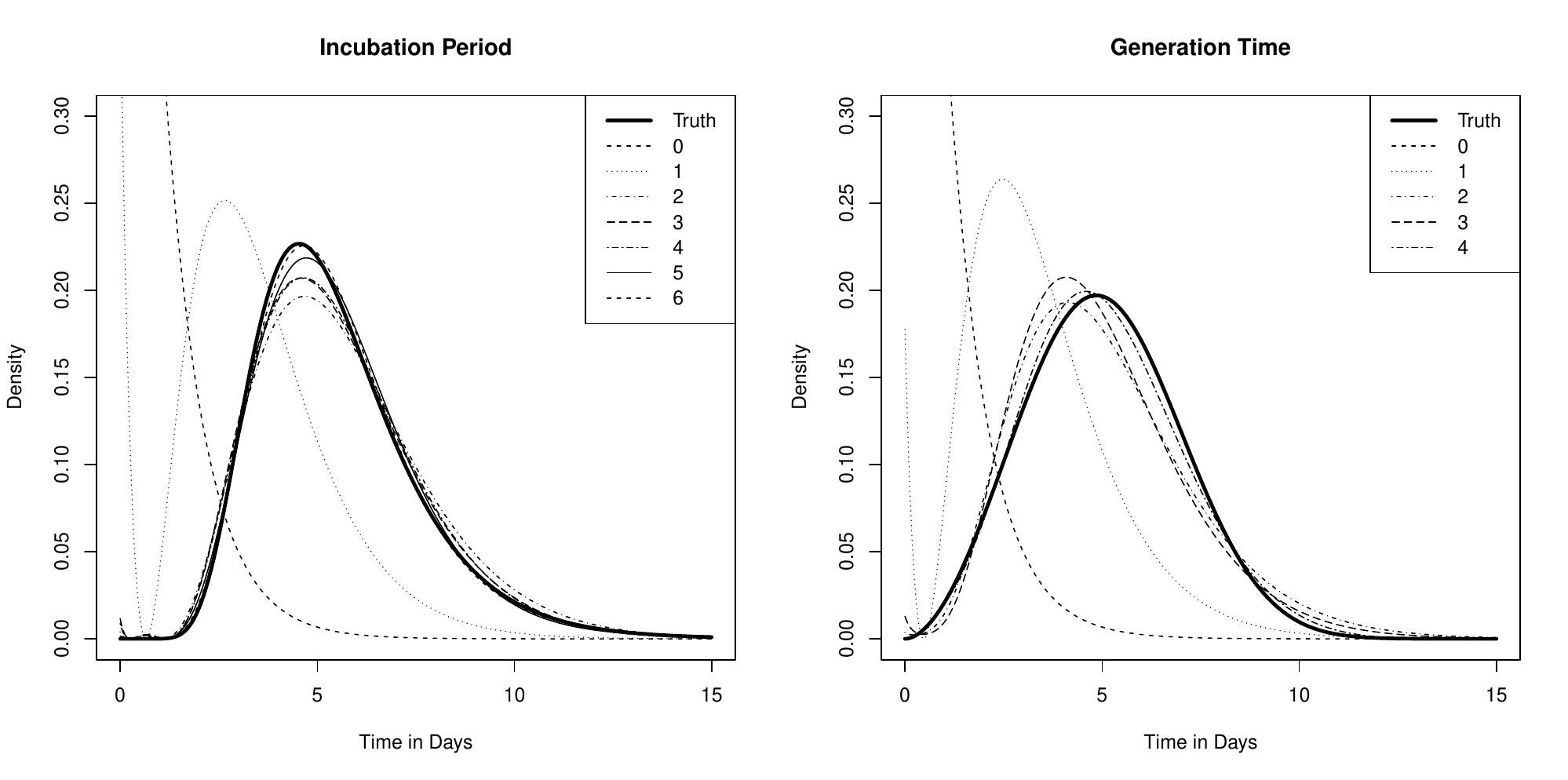}
\caption{Best approximations of $\phi_I$ (left) and $\phi_G$ (right) through  Laguerre densities of the form \eqref{eq:def_phi_theta} for various choices of $m_1$ and $m_2$.}
\label{fig:approx}
\end{figure}

In order to determine the degree of the approximating Laguerre polynomials, we use the information criterion BIC. Since the computations are quite intensive we do the model selection only for one dataset and choose the selected model for all subsequent repetitions. The resulting values of the BIC are shown in Table \ref{tab:BIC_synthetic} and it can be seen that the minimal value is obtained for $m_1=2$ and $m_2=2$. For the remainder of this Section, we use these degrees.

\begin{table}
\centering
\begin{tabular}{llllll}
\toprule
     &   & \multicolumn{4}{c}{$m_2$}        \\
     &   &    1    & 2     &  3     &   4    \\
\toprule
\multirow{4}{*}{$m_1$} & 1 & 486.77 & 421.91 & 418.82 & 422.37 \\ 
                       & 2 & 413.26 & 395.89 & 399.58 & 403.27 \\
                       & 3 & 416.95 & 397.70 & 401.10 & 404.55 \\
                       & 4 & 418.39 & 401.02 & 404.71 & 408.19 \\
\bottomrule
\end{tabular}
\caption{BIC values for one synthetic dataset}
\label{tab:BIC_synthetic}
\end{table}

We simulate now $N=1,000$ datasets according to the above specified data generating process and fit our model to it. Figure \ref{fig:n40_densities} shows the true densities as solid blue lines and corresponding closest Laguerre densities \eqref{eq:def_phi_theta} when choosing degrees $m_1=m_2=2$, respectively. The shaded areas show point-wise confidence bands (constructed based on simulations), e.g. $90\%$ of the estimators lie point-wise in the confidence band of level $90\%$. Different confidence levels ($99$, $95$, $90$, $85$, $80$, $75$, $70$, $65$, $60$) are indicated by different intensities of gray (from light gray to black). It can be seen that the estimation works visibly quite well for the incubation period. The height of the mode is underestimated but this comes from the fact that fitting an order two Laguerre density cannot do better. The estimation for the generation time works a bit less good but still the general trend is captured well by most estimators if we keep in mind that we have here only $40$ observations of a heavily convolved variable. In order to assess the fit of the non-parametric estimator a bit more formally, we compare the estimators to the true densities in terms of the squared Hellinger-distance $\rho_H^2$. The resulting histogram is shown in Figure \ref{fig:n40_hellinger}. Note that squared Hellinger distances are bounded from above by $2$. It can be seen that the estimation of the Incubation Period works better than the estimation of the Generation Time. This is not surprising because we observe transmission pairs each of which contains two independent realisations of incubation periods but only one generation time. Overall the fit of the generation time appears to be fairly good ($90\%$ of the distances are smaller than $0.103$ and $50\%$ of the distances are smaller than $0.034$).

\begin{figure}
\includegraphics[width=\textwidth]{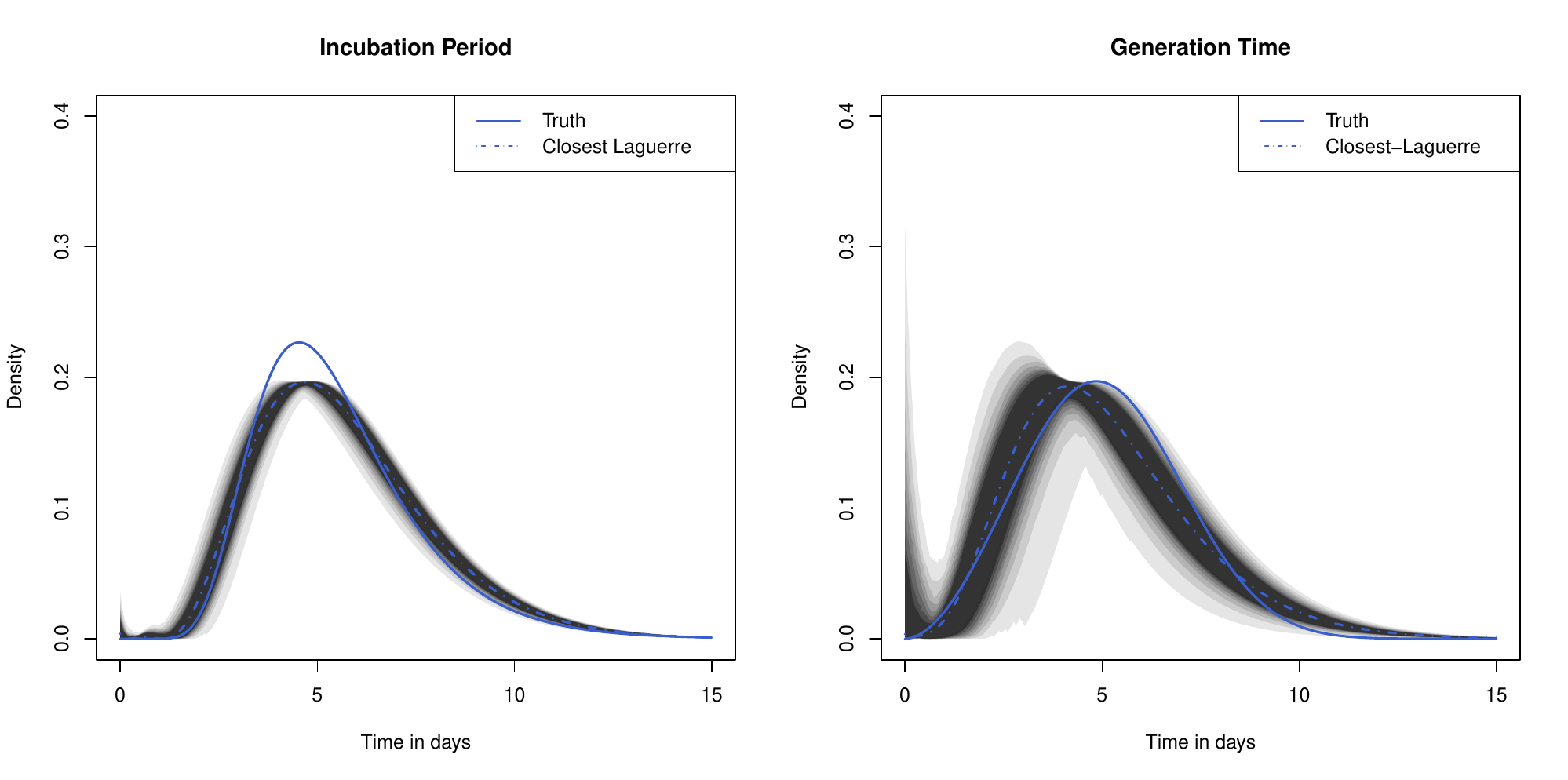}
\caption{Estimation results for simulated data with $n=40$ observations. True densities are shown as solid blue lines and the dashed lines show the Laguerre densities which come closest to the true densities. The shaded areas show the point-wise simulated confidence areas from $99\%$ (lightest gray) over $95\%$ in $5\%$-steps to $60\%$ in black.}
\label{fig:n40_densities}
\end{figure}

\begin{figure}
\includegraphics[width=\textwidth]{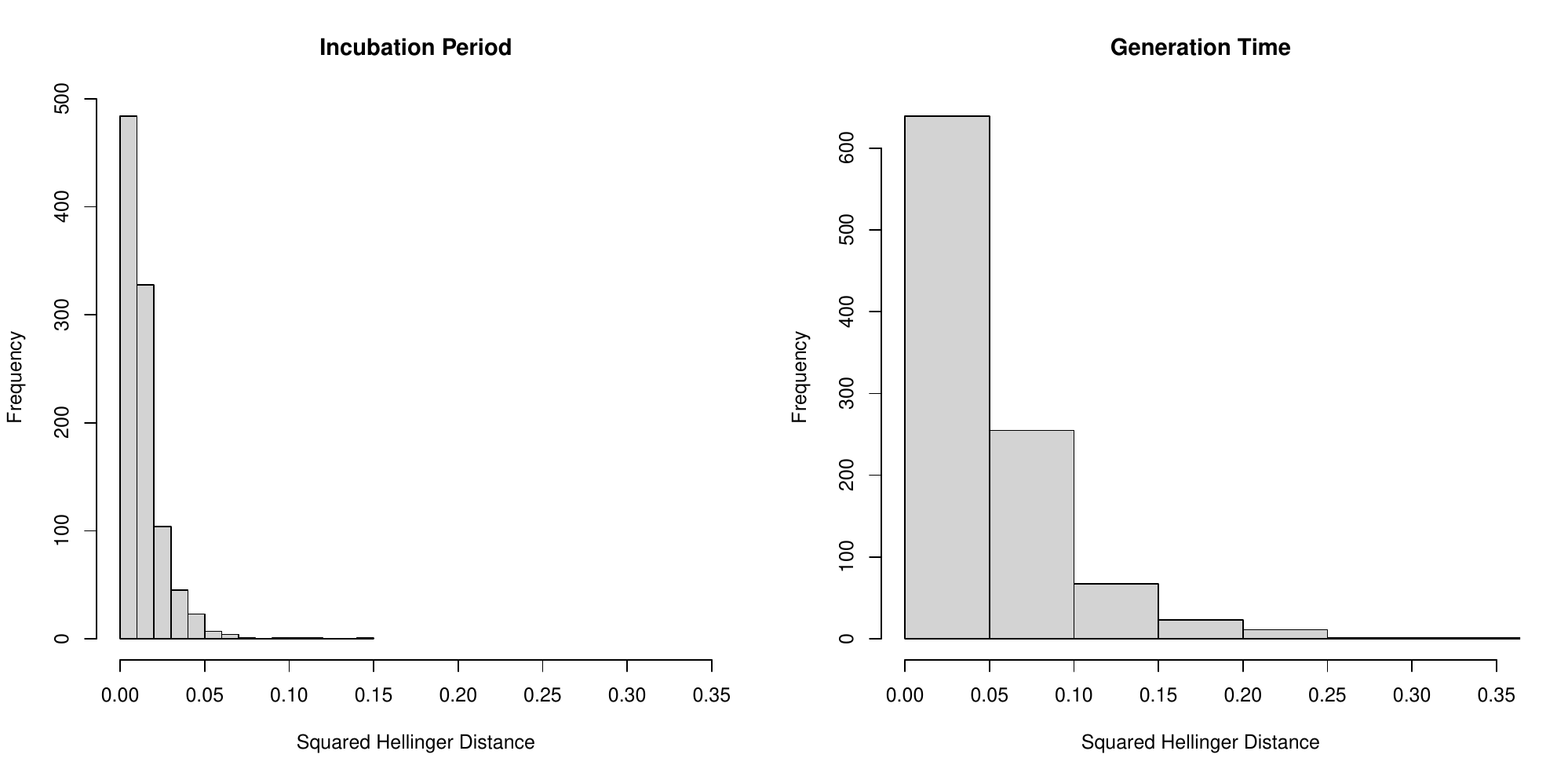}
\caption{Histograms of squared Hellinger distances of estimates to true densities.}
\label{fig:n40_hellinger}
\end{figure}

In order to illustrate that the corresponding plug-in feature estimates enjoy asymptotic normality properties, we consider estimation of the basic reproduction number of a fictional pandemic and estimation of quantiles. In order to estimate the basic reproduction number, we plug the true and the estimated generation time densities into equation \eqref{eq:defF} and take the inverse of it (in this fictional pandemic we choose $r=\log(2)/5$ which means that the case numbers double every five days). The histogram of the estimates is shown in Figure \ref{fig:R0_estimation}. We emphasize that these are simulated results and we cannot draw any conclusions about COVID-19. We can see that the estimation appears to be almost unbiased even in finite samples. This is remarkable because the estimator is based on a non-parametric estimator which is biased (we do not control for over- or under-smoothing). Moreover, the approximation through a normal distribution seems to be reasonably accurate. This motivates further research in establishing a formal asymptotic normality result for this type of semi-parametric inference. Figure \ref{fig:n40_03quant} shows the histograms of the estimated 30\%-quantiles. We can see that the estimates appear to have a bias in finite samples (the dashed lines indicate the true quantiles). This bias seems to stem from the fact that the approximation through low dimensional Laguerre polynomials is not perfect because the estimates centralize around the quantiles of the best Laguerre approximations (dotted lines). This effect remains true for other quantiles which are reported in Appendix \ref{subsubsec:appendix_simstud}.

\begin{figure}
\centering
\includegraphics[width=0.5\textwidth]{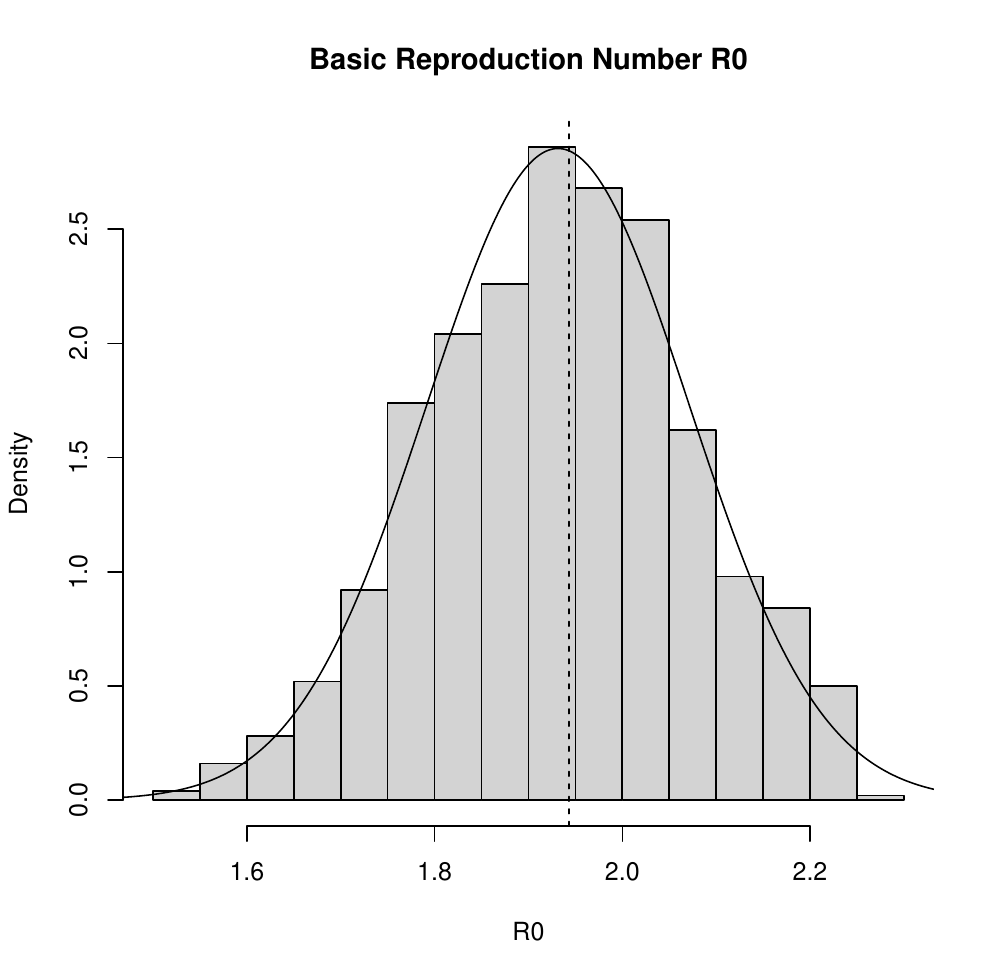}
\caption{Histograms of estimated basic reproduction numbers of fictional pandemic. The dotted line indicates the true $R_0$ and the solid line is the density of a normal distribution with the sample mean and standard deviation.}
\label{fig:R0_estimation}
\end{figure}

\begin{figure}
\centering
\includegraphics[width=\textwidth]{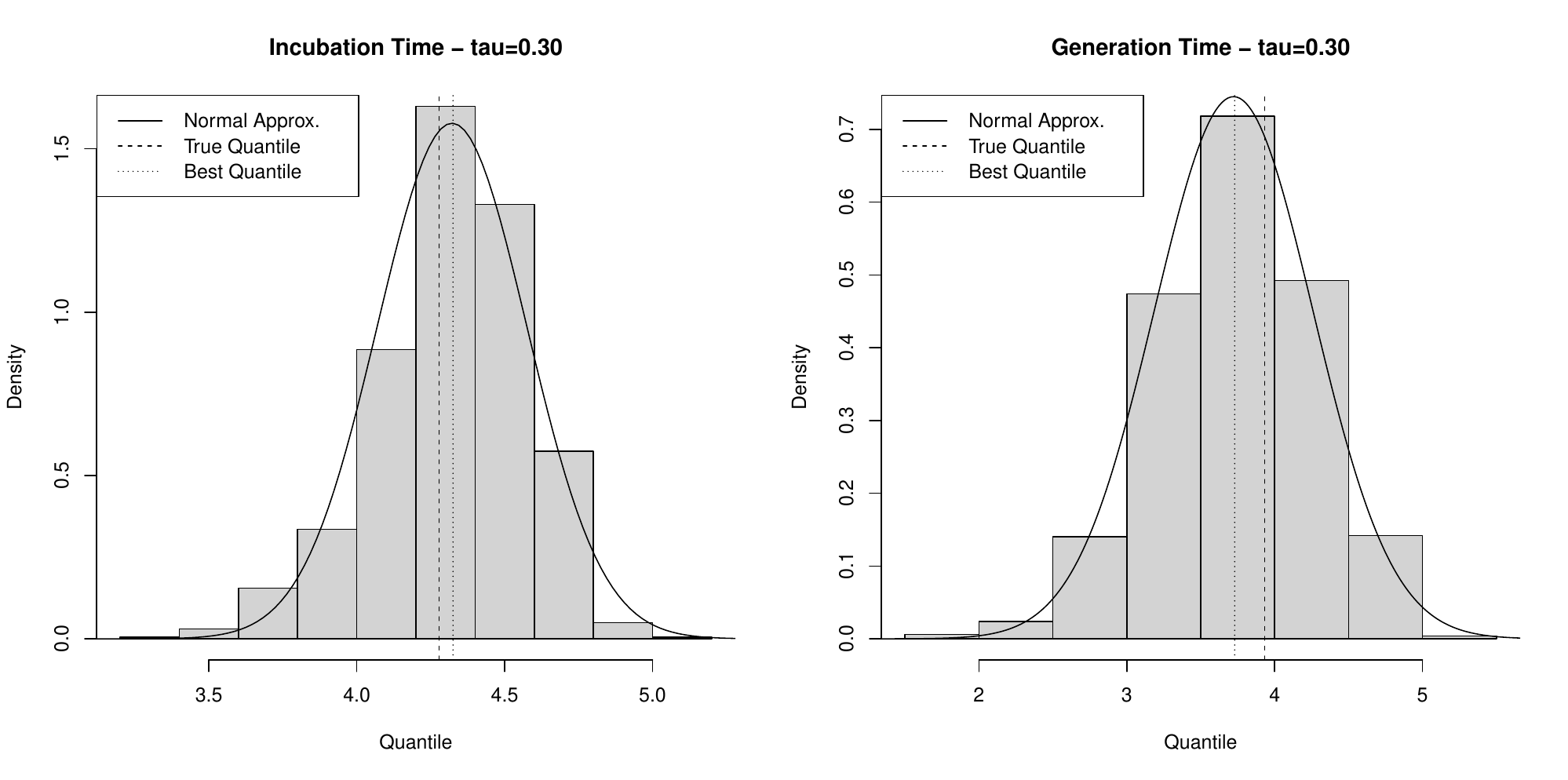}
\caption{Histograms of estimated 30\%-quantile for the incubation time (left) and generation time (right). The dashed lines indicate the respective true values, the dotted lines the quantiles of the best approximating densities and the solid line is the density of a normal distribution with the corresponding sample mean and standard deviation.}
\label{fig:n40_03quant}
\end{figure}

When inspecting the BIC values in Table \ref{tab:BIC_synthetic} we see that the next smallest scores are obtained for the models $(m_1=2,m_2=3)$ and $(m_1=3,m_2=2)$. In order to check the robustness of the model selection, we also inspect the next smallest BIC model, i.e. $(m_1=3,m_2=2)$. The results are similar and are shown in Appendix \ref{subsubsec:appendix_simstud}.

\subsection{Real Data Application}
\label{subsec:real_data}
Finally we apply our methodology to a dataset containing 191 transmission pairs which has been studied by \citet{HMT21,FLWZ20} and can be downloaded from \href{https://elifesciences.org/articles/65534/figures#content}{https://elifesciences.org/articles/65534/ figures\#content} (see the source data corresponding to Figure 2). The data-set is a compilation of five data-sets: \citet{FWK20,HLW20,XLL20,CJL20,ZLW20}. In all data-sets the authors had access to or collected data on transmission pairs. In all cases the authors ensure that the transmission pairs are indeed true transmission pairs e.g. by examining contact and travel histories or quarantines of the involved people. The data was collected with different targets concerning the transmissibility of SARS-COV-2. Thus we can reasonably illustrate our methodology on this dataset: Semi-parametric estimation of features of the generation time and the incubation time.

The dataset contains symptom onset dates for all transmission pairs, but the exposure window is not always reported. In that case we impute the dataset in the same way as \citet{FWK20}: The beginning of the exposure window is at the earliest 60 days before symptom onset of the infector. The end of the infection window is at the latest the symptom onset time of any of the two people in the pair or the end of the exposure window of the second person which is reported in some cases. Since the data is discrete, i.e. we know only the days of symptom onset rather than the exact time, we suppose that the exact time is uniformly distributed throughout the day. Therefore we add a uniform random time between $0$ and $24$ hours to the symptom onset times and exposure window end points. We stress that our interest lies in the theoretical analysis of the methodology and we provide here an illustration for how our methodology can be used. A complete data analysis would for example also require a robustness analysis against potential issues like the question whether the transmission pairs are random samples from the pandemic.

In the following our aim is to use our semi-parametric approach to construct a test whether the parametric fit suggested by \citet{FWK20} is appropriate for the data. Let $\phi_{I,0}$ and $\phi_{G,0}$ denote the densities as defined in Section \ref{subsec:simulation}. More formally, we would like to test the hypotheses
$$H_0^{(I)}: \phi_I=\phi_{I,0}\quad\textrm{and}\quad H_0^{(G)}:\phi_G=\phi_{G,0}.$$
As test statistics we consider $\rho_H(\hat{\phi}_{I,n},\tilde{\phi}_{I,0})^2$ and $\rho_H(\hat{\phi}_{G,n},\tilde{\phi}_{G,0})^2$, where $\tilde{\phi}_{I,0}$ and $\tilde{\phi}_{G,0}$ denote the best approximations with respect to the Hellinger distance of $\phi_{I,0}$ and $\phi_{G,0}$ through Laguerre polynomials, respectively. In order to choose the degrees of the approximation we use in the same way as before the BIC. The resulting values are shown in Table \ref{tab:BIC_real} and it can be seen that $m_1=4$ and $m_2=3$ yields the smallest BIC. Figure \ref{fig:approx} shows that $m_1=4$ yields already a good approximation to $\phi_{I,0}$, similarly we see that $m_2=3$ allows a good approximation of $\phi_{G,0}$. However, in both cases, the representation is not perfect. Therefore, it is very important that we compare the estimate with the closest Laguerre-type density $\tilde{\phi}_{G,0}$ rather than with $\phi_{G,0}$ directly.

\begin{table}
\centering
\begin{tabular}{llllllll}
\toprule
 &   & \multicolumn{6}{c}{$m_2$}                                 \\
 &   & 1       & 2       & 3       & 4       & 5       & 6       \\ 
\toprule
\multirow{6}{*}{$m_1$} & 1 & 1572.22 & 1474.16 & 1460.14 & 1455.04 & 1460.29 & 1459.01 \\ 
 & 2 & 1452.06 & 1423.16 & 1408.44 & 1413.69 & 1418.70 & 1417.45 \\ 
 & 3 & 1436.92 & 1409.72 & 1411.41 & 1412.70 & 1417.56 & 1421.27 \\ 
 & 4 & 1424.08 & 1408.84 & 1405.98 & 1410.06 & 1415.09 & 1415.36 \\ 
 & 5 & 1425.95 & 1412.56 & 1408.15 & 1413.13 & 1418.06 & 1420.15 \\ 
 & 6 & 1428.97 & 1414.82 & 1413.40 & 1417.10 & 1422.33 & 1424.50 \\ 
\bottomrule
\end{tabular}
\caption{BIC values for transmission pair data}
\label{tab:BIC_real}
\end{table}

In order to assess the distribution of the test statistic under $H_0^{(I)}$, $H_0^{(G)}$ or $H_0^{(I)}\cap H_0^{(G)}$, we suppose that the densities of the incubation period and generation time are given by $\phi_{I,0}$ and $\phi_{G,0}$, respectively, and simulate data accordingly including the rounding to complete days and adding uniform noise (note that when testing for $H_0^{(I)}$ or $H_0^{G)}$ we suppose that the respective other density is correctly specified). We do this $N=1,000$ times and show in Figure \ref{fig:hellinger_data} the histograms of the squared Hellinger distances between the estimated densities and those Laguerre densities of type \eqref{eq:def_phi_theta} (with $m_1=4$ and $m_2=3$) which lie closest to the true densities. Note, that in Figure \ref{fig:n40_hellinger} we compared the estimates with the true densities, so both figures cannot be compared. The observed squared Hellinger distances from the dataset are $0.05408553$ for the incubation period and $0.03635778$ for the generation time. These values are shown as vertical lines in Figure \ref{fig:hellinger_data}. In our simulations none of the simulated Hellinger distances for the incubation period are larger than the observed distance.  Consequently, in none of the simulated cases both distances are simultaneously larger than the observed distances. For the generation time $44.8\%$ of the simulated distances are larger than the observed distance. We conclude that the data-set we considered here suggests that the suggested parametric fit for the incubation time might not be correct. For the generation time, the data-set contains no evidence which suggests that the parametric fit might be wrong.

\begin{figure}
\includegraphics[width=\textwidth]{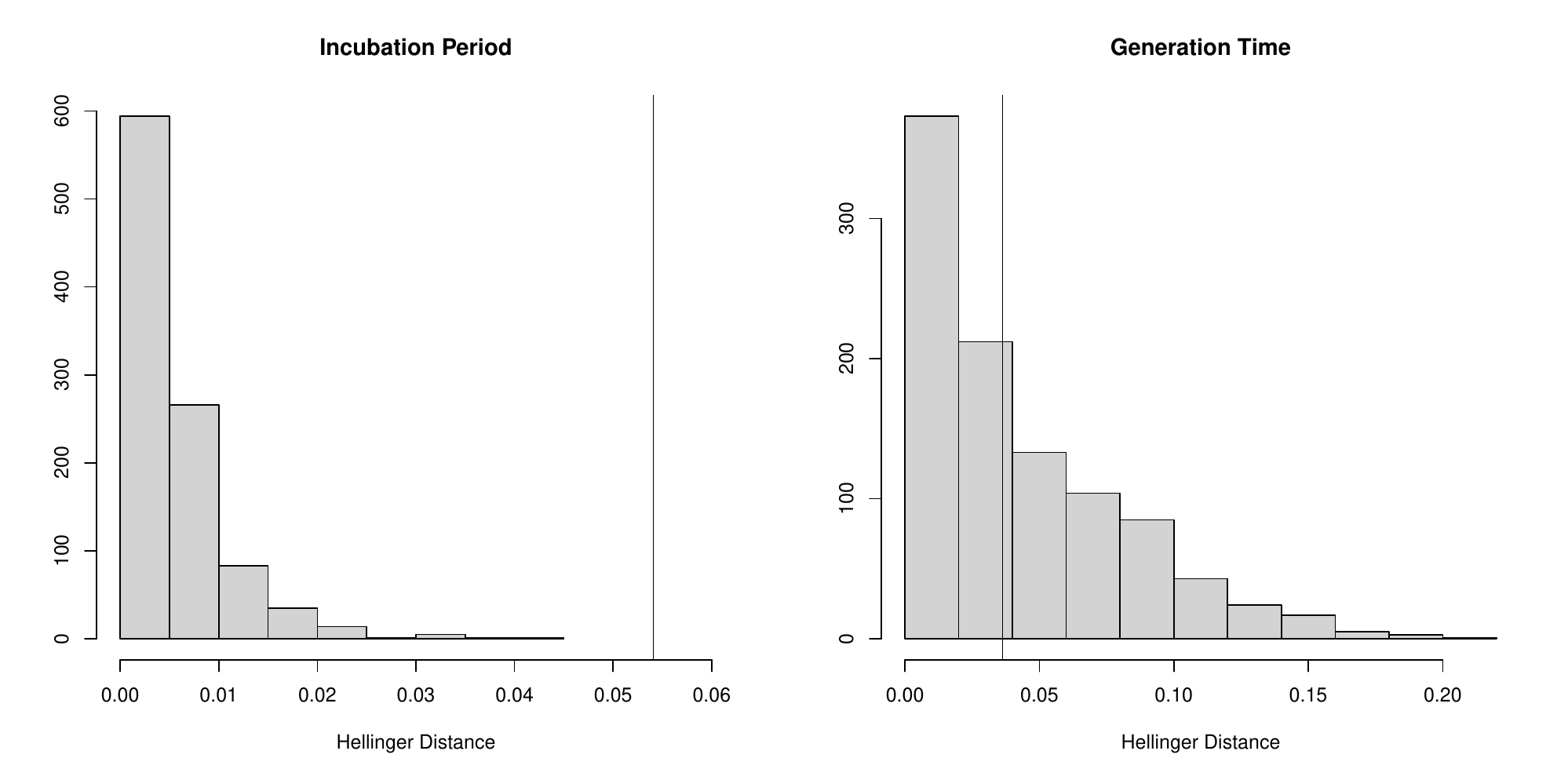}
\caption{Histograms of squared Hellinger distances of $N=1,000$ simulations from the model under $H_0^{(I)}\cap H_0^{(G)}$. Distances are computed between the estimated model and the closest Laguerre type model. The vertical lines indicate the observed values of the test statistic.}
\label{fig:hellinger_data}
\end{figure}

In Table \ref{tab:BIC_real} we see that the BIC values for all choices of $m_1=3,4,5$ and $m_2=2,3,4$ are all very similar. In order to check the robustness of the method to the model selection, we provide in Section \ref{subsubsec:appendix_datapp} in the Appendix a similar analysis for the second smallest BIC: $m_1=5$ and $m_2=3$.

We stress that the above analysis should be understood as a recipe for a data analysis rather than an in-depth analysis of the provided dataset. Other types of semi-parametric analyses like the ones outlined in Section \ref{subsec:simulation} can be implemented in a similar way.

\section{Conclusion and Related Questions}
\label{sec:conclusion}
We have introduced a semi-parametric estimator for the generation time and incubation period from observational data. We have shown that both distributions can be identified from the observations and we presented a simple, consistent semi-parametric estimator which is based on Laguerre polynomials. These results are the first steps for more results in the general realm of semi-parametric inference for epidemics: As specific examples we mention the reproduction number $R_0$ and tests for parametric assumptions. But also the probability of pre-symptomatic infection $\IP(G\leq I_1)$ can be of interest. All these quantities are continuous functions of the densities $\phi_I$ and $\phi_G$ and therefore it can be expected that asymptotic normality results for estimators based on our estimators $\hat{\phi}_{I,n}$ and $\hat{\phi}_{G,n}$ can be proven. However, it should be mentioned that such results for $R_0$ are possibly more challenging because they require another estimator $\hat{r}_n$.

This work can be extended in several directions. It might for example not be clear when symptoms exactly start. Therefore, it might be possible that just a window for symptom onset can be supplied (\citet{LGB20}). Moreover, it can be of interest to include asymptomatic patients by including a certain \emph{cure-probability}, i.e. the probability with which patients will never show symptoms. As an alternative one could also consider follow up studies in which the symptom onset time of patients is considered as a censored variable (in which case asymptomatic patients can be interpreted as patients who show symptoms at $\infty$). Finally, the dataset by \citet{FWK20} includes also a window for infection of the second person. It would of course be interesting to include this information in the model. But we would like to point out that this is not entirely trivial because it is unclear how to model the time of infection within this window. As we motivated in the discussion after Remark \ref{rem:location} a simple uniformity assumption is possibly difficult to justify. Therefore, we would suggest to include this distribution in the estimation in a suitable way. In theoretic terms, the most interesting question would certainly be to establish asymptotic normality results which allow the researcher to make quantitative statements. Such statements can for example be achieved by following the semi-parametric framework as e.g. in \citet{S97}. Another interesting question would be how to adequately incorporate covariates in the model. It could be the case that e.g. younger and older people have different incubation periods or generation times. Finally, in a different branch of literature, one tries to avoid the assumption of independence between symptom onset and infection and rather models the infection time relative to the symptom onset (cf. \citet{HMT21,FLWZ20,CBK20}). Our methodology can be applied to this setting as well with some adjustments of the theory.

\section*{Acknowledgements}
We are grateful to Niel Hens for useful discussions. Moreover we thank an unknown associate editor and two anonymous referees for reading and commenting our paper. Their remarks have lead to a great improvement of the paper.

\newpage
\bibliography{mybib}{}
\bibliographystyle{plainnat}

\newpage
\section{Appendix}
\label{sec:appendix}

\subsection{A Consistency Result and Approximation Through Laguerre Polynomials}
\label{subsec:WongShen}
In this section we state two results from the literature which are relevant for this paper.

The following Theorem is just a re-formulation for a special case of Theorem 4 in \citet{WS95} which is stated here for the convenience of the reader: In their paper the authors study approximate sieve estimation, i.e., they allow that the estimator only approximately maximizes the likelihood. In their notation this means that we assume in the present paper that $\eta_n=0$, this is already included in the following version of Theorem 4 of \citet{WS95}.

\begin{theorem}
\label{thm:WS95}
Let $Y_1,...,Y_n$ be iid observations which have a density $p_0$. Let moreover $\mathcal{F}_n$ be an arbitrary sequence of sieve spaces for density estimation and let $\hat{p}_n$ denote the sieve-MLE. Suppose that there are constants $c_1,c_2>0$ and a sequence $\epsilon_n>0$ such that
$$\int_{\frac{\epsilon_n^2}{2^8}}^{\sqrt{2}\epsilon_n}\sqrt{\log\mathcal{N}_{[]}\left(\frac{u}{c_1},\mathcal{F}_n,\rho_H\right)}du\leq c_2\sqrt{n}\epsilon_n^2.$$
Let $\delta_n(\alpha):=\inf_{q\in\mathcal{F}_n}\rho_{\alpha}(p_0,q)$ for $\alpha\in(0,1]$. Suppose that there is $\alpha\in(0,1]$ such that $\delta_n(\alpha)<1/\alpha$. Then, there is a constant $c>0$ (which depends on the model) such that for
$$\epsilon_n^*(\alpha):=\max\left(\epsilon_n,\sqrt{\frac{4\delta_n(\alpha)}{c}}\right)$$
we have for another constant $C>0$
$$\IP\left(\rho_H(\hat{p}_n,p_0)\geq\epsilon_n^*(\alpha)\right)\leq5\exp\left(-Cn\epsilon_n^*(\alpha)^2\right)+\exp\left(-\frac{1}{4}n\alpha c\epsilon_n^*(\alpha)^2\right).$$
\end{theorem}

The next result is a statement about approximating functions by using Laguerre polynomials. The following is a combination of Theorem 1 and Remark 2 in Chapter II.8 of \citet{NU88}. We formulate this theorem here in our setting. The original statement is more general.

\begin{theorem}
\label{thm:lag_approx}
Let $f:[0,\infty)\to\IR$ be continuous and have a piecewise continuous derivative $f'$. Consider the series
$$f_N(x):=\sum_{k=0}^Nc_kL_k(x),\quad c_k:=\int_0^{\infty}f(x)L_k(x)e^{-x}dx.$$
If the integrals
$$\int_0^{\infty}f(x)^2e^{-x}dx\,\textrm{ and }\,\int_0^{\infty}f'(x)^2xe^{-x}dx$$
converge, we have $f_N\to f$ for $N\to\infty$ uniformly over compact sets $[x_1,x_2]\subseteq(0,\infty)$ and $\lim_{N\to\infty}\int_0^{\infty}\left[f(x)-f_N(x)\right]^2e^{-x}dx\to0$.
\end{theorem} 

\subsection{Further Empirical Results}
\subsubsection{Simulation Study}
\label{subsubsec:appendix_simstud}

In this Section we present additional simulation results complementing the results from Section \ref{subsec:simulation}. Figures \ref{fig:n40_05quant}-\ref{fig:n40_09quant} show histograms for estimation of the 50\%-, 70\%- and 90\%-quantiles. The results are very similar to the results for the 30\%-quantile which are discussed in Section \ref{subsec:simulation}.

\begin{figure}
\centering
\includegraphics[width=\textwidth]{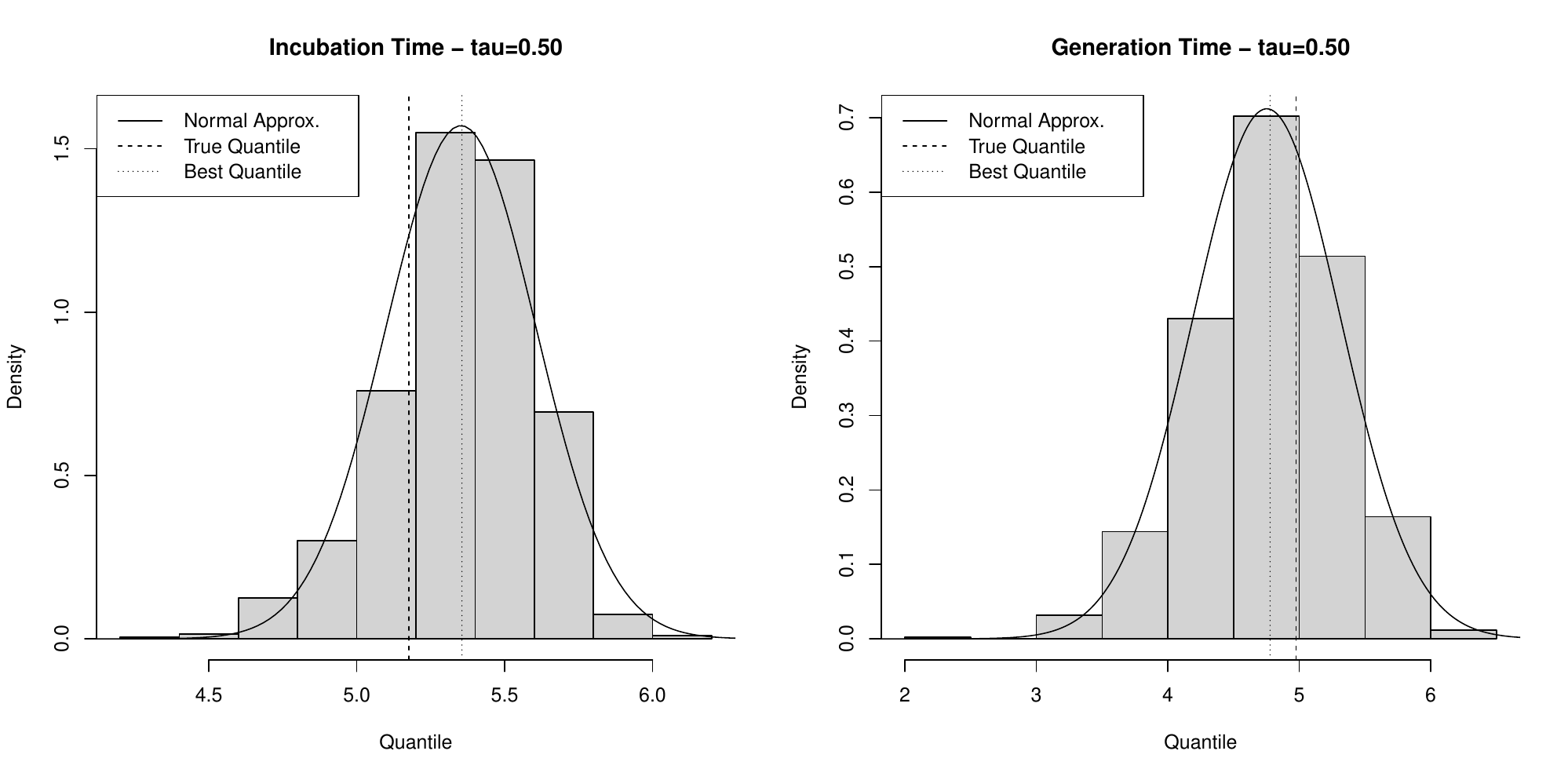}
\caption{Histograms of estimated 50\%-quantile for the incubation time (left) and generation time (right). The dashed lines indicate the respective true values, the dotted lines the quantiles of the best approximating densities and the solid line is the density of a normal distribution with the corresponding sample mean and standard deviation.}
\label{fig:n40_05quant}
\end{figure}

\begin{figure}
\centering
\includegraphics[width=\textwidth]{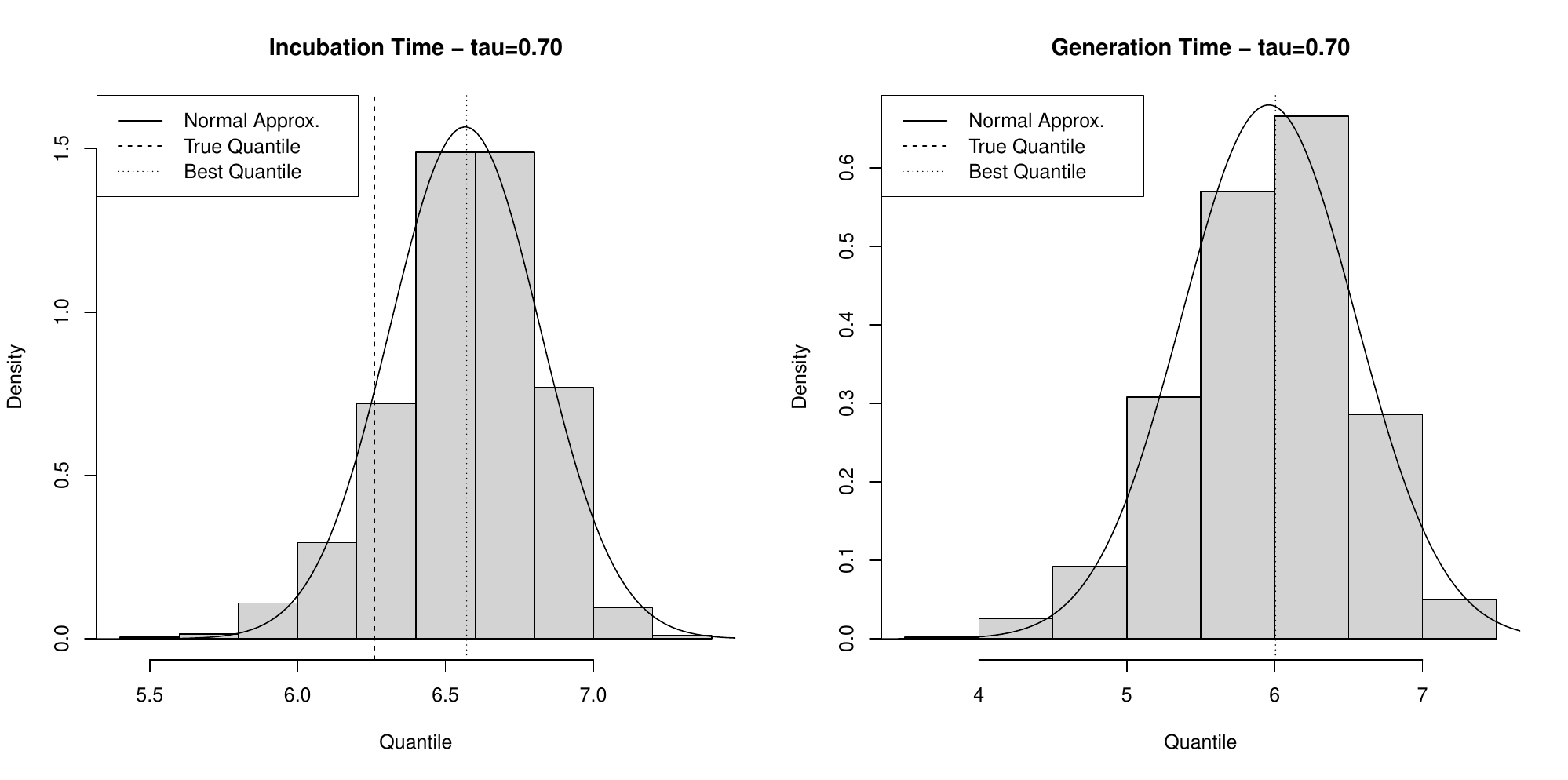}
\caption{Histograms of estimated 70\%-quantile for the incubation time (left) and generation time (right). The dashed lines indicate the respective true values, the dotted lines the quantiles of the best approximating densities and the solid line is the density of a normal distribution with the corresponding sample mean and standard deviation.}
\label{fig:n40_07quant}
\end{figure}

\begin{figure}
\centering
\includegraphics[width=\textwidth]{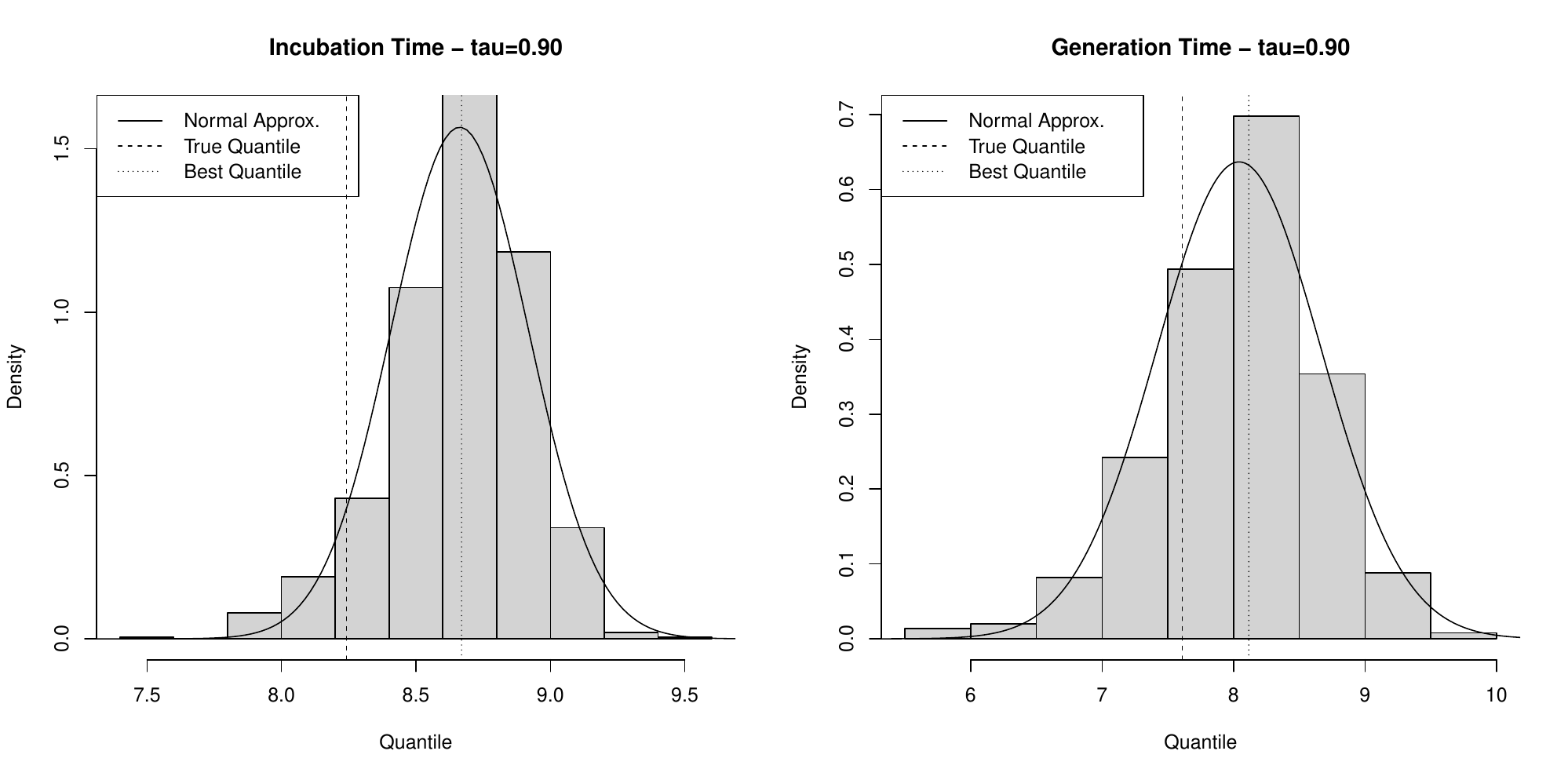}
\caption{Histograms of estimated 90\%-quantile for the incubation time (left) and generation time (right). The dashed lines indicate the respective true values, the dotted lines the quantiles of the best approximating densities and the solid line is the density of a normal distribution with the corresponding sample mean and standard deviation.}
\label{fig:n40_09quant}
\end{figure}

Next, we use the same set-up as in Section \ref{subsec:simulation} in the main text, however, here we choose as a model $m_1=3$ and $m_2=2$, i.e., the model has now more flexibility for the incubation time. The estimation results are visualized in Figure \ref{fig:n40_densities_23} and the difference in the squared Hellinger distance is shown in Figure \ref{fig:n40_hellinger_23}. In Figure \ref{fig:n40_densities_23} it appears that for the incubation time the mode moves closer to its true location, it is sometimes even overestimated (compared to the case $m_2=2$ which was shown in Section \ref{subsec:simulation}). Moreover, the estimated incubation time densities seem to fluctuate more, i.e., they show a higher variance due to the higher flexibility. The estimates for the generation time appear to be almost identical. In terms of the Hellinger distances, cf. histograms in Figure \ref{fig:n40_hellinger_23}, the results appear to be very similar to the results obtained in Section \ref{subsec:simulation}. In general we see that estimation with $n=40$ observations yields reasonable results for both model complexities. However, we expect that the estimation can be improved if more observations are available enabling the method to choose better approximations.

\begin{figure}
\includegraphics[width=\textwidth]{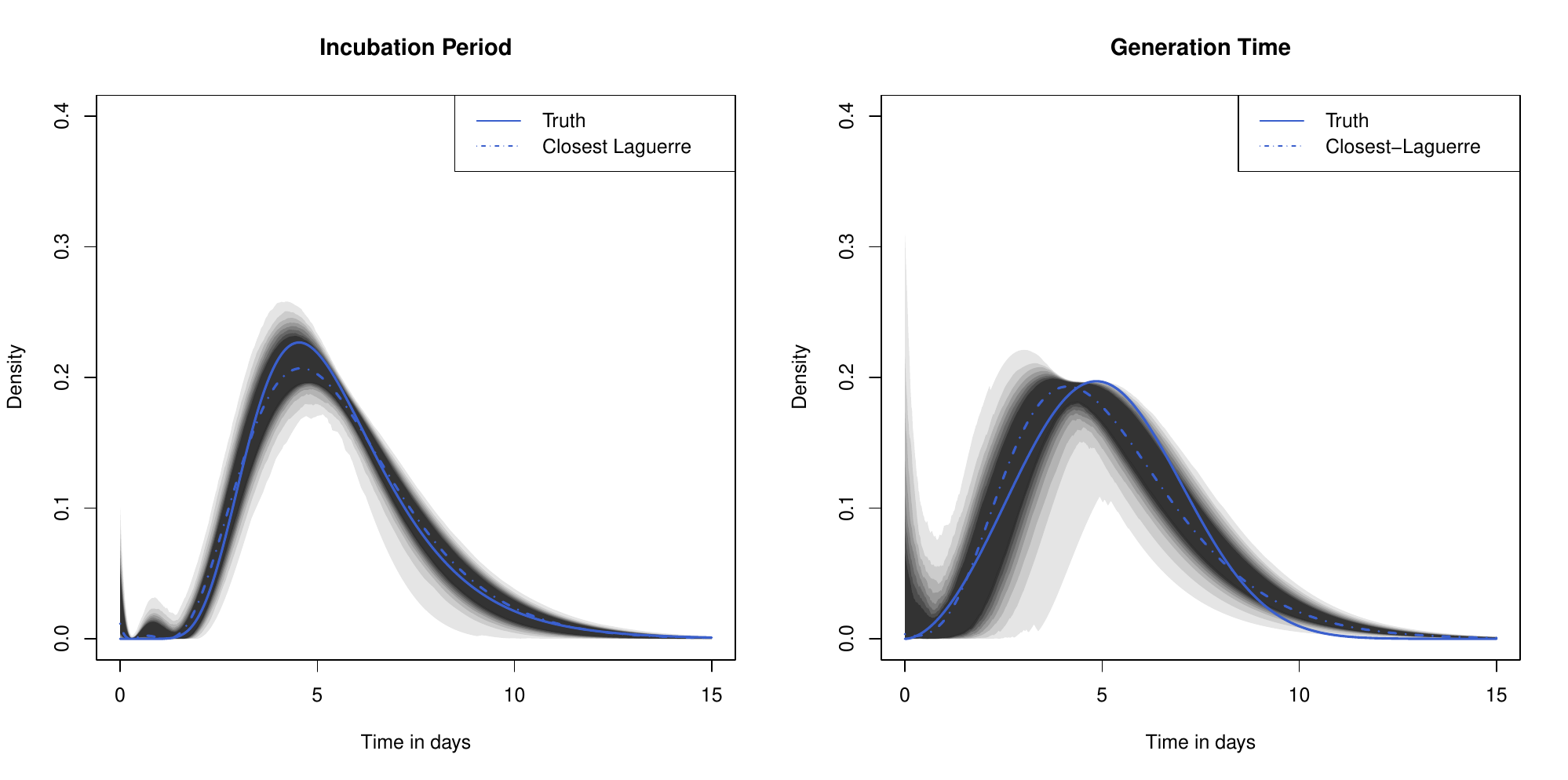}
\caption{Estimation results for simulated data with $n=40$ observations. True densities are shown as solid blue lines and the dashed lines show the Laguerre densities which come closest to the true densities. The shaded areas show the point-wise simulated confidence areas from $99\%$ (lightest gray) over $95\%$ in $5\%$-steps to $60\%$ in black. Model choices: $m_1=2$ and $m_2=3$}
\label{fig:n40_densities_23}
\end{figure}

\begin{figure}
\includegraphics[width=\textwidth]{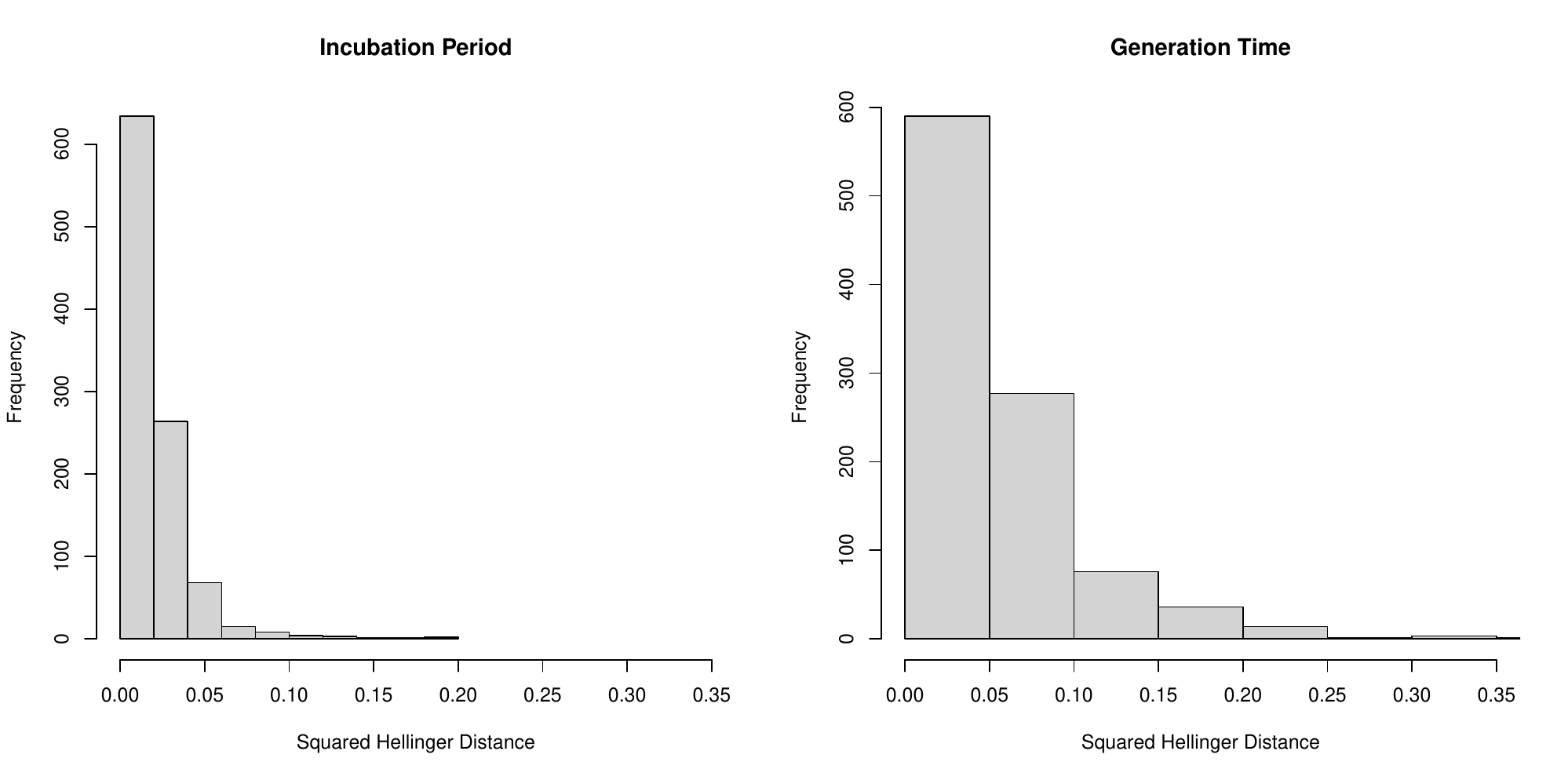}
\caption{Histograms of squared Hellinger distances of estimates to true densities. Model Choices: $m_1=3$ and $m_2=2$.}
\label{fig:n40_hellinger_23}
\end{figure}

\subsubsection{Real Data Application}
\label{subsubsec:appendix_datapp}

In this section we show an analysis similar to that from Section \ref{subsec:real_data} in the main text, but here we choose $m_1=5$ and $m_2=3$. We begin with a simulation: We generate $N=1,000$ datasets assuming that the models specified in $H_0^{(I)}$ and $H_0^{(G)}$ are correct. Figure \ref{fig:hellinger_53_data} shows the squared Hellinger distances of the estimates to the closest Laguerre type densities with $m_1=5$ and $m_2=3$. The vertical lines show the distances which are obtained from the real data set. In case of the incubation period only $1\%$ of the simulated distances are larger, for the generation time this percentage is higher, $44.9\%$. Finally, both distances are simultaneously larger for both densities in $0.3\%$ of the simulated cases. Thus, we conclude that also in this larger model class there is no evidence in the dataset which would contradict the parametric model for the generation time fitted by \citet{FWK20}. However the parametric model for the incubation time used by \citet{FWK20} and \citet{LGB20} might be questionable depending on the desired level of the test.

\begin{figure}
\includegraphics[width=\textwidth]{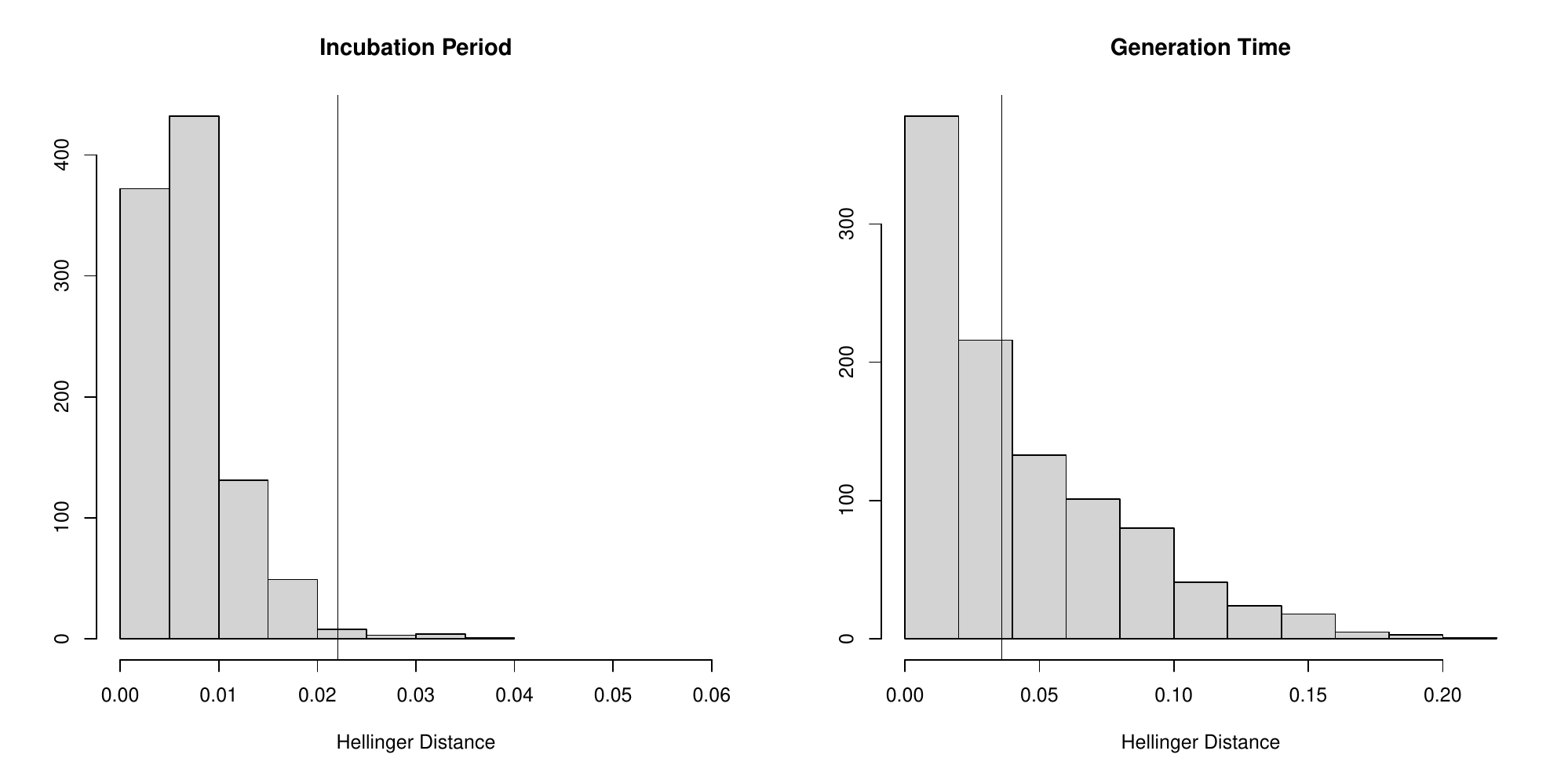}
\caption{Histogram of squared Hellinger distances of $N=1,000$ simulations from the model under $H_0^{(I)}\cap H_0^{(G)}$. Distances are computed between the estimated model and the closest Laguerre type model (for $m_1=5$ and $m_2=3$). The vertical lines indicate the observed values of the test statistic.}
\label{fig:hellinger_53_data}
\end{figure}

\subsection{Proofs of Section \ref{subsec:discussion}}
\label{subsec:cor_proofs}

\begin{proof}[Proof of Corollary \ref{cor:consistency}]
For the proof we make the following definitions: Let $\Phi_I$, $\Phi_G$ and $\Phi_{S_1,S_2,W,C}$ be the characteristic functions of the corresponding random variables, i.e., for real numbers $x,x_1,...,x_4\in\IR$ (note that in the definition of $\phi_I$ it doesn't matter whether we take $I_1$ or $I_2$)
\begin{align*}
\Phi_I(x)&:=\IE\left(e^{ixI_1}\right), \quad \Phi_G(x):=\IE\left(e^{ixG}\right), \\
\Phi_{S_1,S_2,W,C}(x_1,x_2,x_3,x_4)&:=\IE\left(e^{i(x_1S1+x_2S_2+x_3W+x_4C)}\right).
\end{align*}
As a first step, we compute $\Phi_{S_1,S_2,W,C}$. We have by definition of $(S_1,S_2,W,C)$ and independence
\begin{align*}
&\Phi_{S_1,S_2,W,C}(x_1,x_2,x_3,x_4)=\IE\left(e^{i(x_1+x_2)T_1}e^{ix_1I_1}e^{ix_2I_2}e^{ix_2G}e^{ix_3W}e^{ix_4C}\right) \\
=&\IE\left(e^{ix_3W+ix_4C}\IE\left(e^{i(x_1+x_2)T_1}\big|W,C\right)\right)\Phi_I(x_1)\Phi_I(x_2)\Phi_G(x_2) \\
=&\Psi(x_1+x_2,x_3,x_4)\Phi_I(x_1)\Phi_I(x_2)\Phi_G(x_2),
\end{align*}
where
$$\Psi(x,z_1,z_2):=\IE\left(e^{iz_1W+iz_2C}\IE\left(e^{ixT_1}\big|W,C\right)\right).$$
Let now $z_1(x),z_2(x)$ be as in Assumption (C). We obtain the following two equations
\begin{align*}
\Phi_{S_1,S_2,W,C}\left(x,0,z_1(x),z_2(x)\right)&=\Psi(x,z_1(x),z_2(x))\Phi_I(x), \\
\Phi_{S_1,S_2,W,C}\left(0,x,z_1(x),z_2(x)\right)&=\Psi(x,z_1(x),z_2(x))\Phi_I(x)\Phi_G(x).
\end{align*}
Note that Assumption (C) implies that $\Phi_{S_1,S_2,W,C}(x,0,z_1(x),z_2(x))\neq0$ almost everywhere. All of this, in turn, implies the relations (almost everywhere by Assumption (C))
\begin{align}
\Phi_I(x)=&\Psi(x,z_1(x),z_2(x))^{-1}\Phi_{S_1,S_2,W,C}\left(x,0,z_1(x),z_2(x)\right), \label{eq:identI}\\
\Phi_G(x)=&\Phi_{S_1,S_2,W,C}\left(x,0,z_1(x),z_2(x)\right)^{-1}\Phi_{S_1,S_2,W,C}\left(0,x,z_1(x),z_2(x)\right). \label{eq:identG}
\end{align}
Let $\hat{C}_n,\hat{W}_n,\hat{T}_{1,n},\hat{I}_{1,n},\hat{I}_{2,n}$ and $\hat{G}_n$ be random variables which have the relation as specified in \eqref{eq:mod} but where $\phi_I$ is replaced by $\hat{\phi}_{I,n}$ and $\phi_G$ is replaced by $\hat{\phi}_{G,n}$. Let moreover, $\hat{S}_{1,n}$ and $\hat{S}_{2,n}$ be defined analogously to their corresponding quantities $S_1$ and $S_2$ with corresponding characteristic functions $\Phi_{\hat{S}_{1,n},\hat{S}_{2,n},\hat{W}_n,\hat{C}_n}$, $\Phi_{\hat{I}_n}$ and $\Phi_{\hat{G}_n}$.

Since convergence in probability implies the existence of an almost surely convergent subsequence (cf. Lemma 1.9.2 in \citet{VW96}), we find by Theorem \ref{thm:consistency} that there is an event $\Omega$ of probability 1 and a subsequence which we indicate again by $f_{\hat{\phi}_{I,n},\hat{\phi}_{G,n}}$ such that
$$\rho_H\left(f_{\phi_I,\phi_G}p_C,f_{\hat{\phi}_{I,n},\hat{\phi}_{G,n}}p_C\right)\to0$$
for all results in $\Omega$. Theorem 1 in Chapter 3.9 in \citet{S16} implies that then
$$\left|\sum_{c=1}^{K_C}\int r(s_1,s_2,w,c)\left(f_{\phi_I,\phi_G}(s_1,s_2,w|c)-f_{\hat{\phi}_{I,n},\hat{\phi}_{G,n}}(s_1,s_2,w|c)\right)p_C(c)d(s_1,s_2,w)\right|\to0,$$
uniformly over all functions $r:\IR^3\times\{1,...,K_C\}\to\IC$ with $\|r\|_{\infty}\leq1$. Since the functions $r(s_1,s_2,w,c)=\exp(i(x_1s_1+x_2s_2+x_3w+x_4c))$ are bounded for all $x_1,x_2,x_3,x_4\in\IR$, we conclude that $\Phi_{\hat{S}_{1,n},\hat{S}_{2,n},\hat{W}_n,\hat{C}_n}$ converges uniformly to $\Phi_{S_1,S_2,W,C}$. Hence, for almost any fixed $x$ we eventually have $\Phi_{\hat{S}_{1,n},\hat{S}_{2,n},\hat{W}_n,\hat{C}_n}(x,0,z_1(x),z_2(x))\neq0$ and thus eventually analogue versions of the relations \eqref{eq:identI} and \eqref{eq:identG} hold for $\Phi_{\hat{I}_n}$ and $\Phi_{\hat{G}_n}$. As a consequence, we obtain the following point-wise almost everywhere convergences
\begin{align*}
\Phi_{\hat{I}_n}(x)=&\Psi(x,z_1(x),z_2(x))^{-1}\Phi_{\hat{S}_{1,n},\hat{S}_{2,n},\hat{W}_n,\hat{C}_n}\left(x,0,z_1(x),z_2(x)\right) \\
&\quad\quad\to\Psi(x,z_1(x),z_2(x))^{-1}\Phi_{S_1,S_2,W,C}\left(x,0,z_1(x),z_2(x)\right)=\Phi_I(x), \\
\Phi_{\hat{G}_n}(x)=&\Phi_{\hat{S}_{1,n},\hat{S}_{2,n},\hat{W}_n,\hat{C}_n}\left(x,0,z_1(x),z_2(x)\right)^{-1}\Phi_{\hat{S}_{1,n},\hat{S}_{2,n},\hat{W}_n,\hat{C}_n}\left(0,x,z_1(x),z_2(x)\right) \\
&\quad\quad\to\Phi_{S_1,S_2,W,C}\left(x,0,z_1(x),z_2(x)\right)^{-1}\Phi_{S_1,S_2,W,C}\left(0,x,z_1(x),z_2(x)\right)=\Phi_G(x).
\end{align*}
Now by Levy's Theorem (cf. Theorem 1.7.6 \citet{B68}, an inspection of the proof reveals that in the univariate case convergence almost everywhere is a sufficient condition), we conclude that $\hat{F}_{I,n}\to F_I$ and $\hat{F}_{G,n}\to F_G$ point-wise for all realisations in $\Omega$. Since the distribution functions are continuous, these convergences are also uniform. Recall that the above argument holds for a subsequence and all realisations in $\Omega$. But we can repeat the same argument starting from a subsequence and showing in that way that every subsequence contains a sub-subsequence for which the corresponding distribution functions converge almost surely. This implies convergence in probability by Lemma 1.9.2 in \citet{VW96} and the proof is complete.
\end{proof}

\begin{proof}[Proof of Corollary \ref{cor:R}]
We make a similar subsequence of a sub-sequence argument as in Corollary \ref{cor:consistency}: Consider sub-sequences of $\hat{r}_n$ and $\hat{F}_{G,n}$ (which we again do not indicate in the notation) such that $\hat{r}_n\to r$ almost surely and $\left\|\hat{F}_{G,n}-F_G\right\|_{\infty}\to0$ almost surely (the former is possible by assumption and the latter by Corollary \ref{cor:consistency}). The following considerations are made for each realisation in a set of probability one on which these two convergences hold. Choose $n$ so large such that $\hat{r}_n\geq r/2$. Choose moreover, $c>0$ such that $t\exp(-tr/2)\leq c$ for all $t\geq0$. We have
\begin{align}
&\int_0^{\infty}e^{-\hat{r}_nt}\hat{\phi}_{G,n}dt-\int_0^{\infty}e^{-rt}\phi_G(t)dt \nonumber \\
=&\int_0^{\infty}\left(e^{-\hat{r}_nt}-e^{-rt}\right)\hat{\phi}_{G,n}dt+\int_0^{\infty}e^{-rt}\left(\hat{\phi}_{G,n}-\phi_G(t)\right)dt. \label{eq:split1}
\end{align}
By a Taylor expansion we obtain for the first part of the right hand side above and for some $r^*\in[r,\hat{r}_n]$ (note that by assumption $r^*\geq r/2$ and hence $t\cdot \exp(-r^*t)\leq c$)
$$\left|\int_0^{\infty}\left(e^{-\hat{r}_nt}-e^{-rt}\right)\hat{\phi}_{G,n}(t)dt\right|\leq\left|\hat{r}_n-r\right|\int_0^{\infty}te^{-r^*t}\hat{\phi}_{G,n}(t)dt\leq c\left|\hat{r}_n-r\right|\to0.$$
For the second term on the right hand side of \eqref{eq:split1} we obtain via integration by parts
\begin{align*}
&\int_0^{\infty}e^{-rt}\left(\hat{\phi}_{G,n}-\phi_G(t)\right)dt \\
=&e^{-rt}\left(\hat{F}_{G,n}(t)-F_G(t)\right)\Big|_{t=0}^{\infty}+\int_0^{\infty}re^{-rt}\left(\hat{F}_{G,n}(t)-F_G(t)\right)dt
\end{align*}
which converges to zero since $\hat{F}_{G,n}\to F$ uniformly. Now, a subsequence of a subsequence argument completes the proof.
\end{proof}

\subsection{Proofs of Section \ref{subsec:proof}}
\label{subsec:lemma_proofs}

\begin{proof}[Proof of Lemma \ref{lem:rho}]
We firstly apply Hoelder's Inequality with $p=1/(1-\alpha)$ and $q=1/\alpha$ (use that $\alpha\in(0,1)$) to get
\begin{align*}
&\rho_{\alpha}(f_1p_C,f_2p_C) \\
=&\frac{1}{\alpha}\sum_{c=1}^{K_C}\int_{[0,\infty)^3}f_1(x_1,x_2,\omega|c)p_C(c)\left(\left(\frac{f_1(x_1,x_2,\omega|c)p_C(c)}{f_2(x_1,x_2,\omega|c)p_C(c)}\right)^{\alpha}-1\right)d(x_1,x_2,\omega) \\
=&\frac{1}{\alpha}\sum_{c=1}^{K_C}\int_{[0,\infty)^3}\left(\frac{f_1(x_1,x_2,\omega|c)}{f_2(x_1,x_2,\omega|c)}\right)^{\alpha}f_1(x_1,x_2,\omega|c)^{1-\alpha} \\
&\quad\quad\quad\quad\times\left(f_1(x_1,x_2,\omega|c)^{\alpha}-f_2(x_1,x_2,\omega|c)^{\alpha}\right)p_C(c)d(x_1,x_2,\omega) \\
\leq&\frac{1}{\alpha}\left(\sum_{c=1}^{K_C}\int_{[0,\infty)^3}\left(\frac{f_1(x_1,x_2,\omega|c)}{f_2(x_1,x_2,\omega|c)}\right)^{\frac{\alpha}{1-\alpha}}f_1(x_1,x_2,\omega|c)p_C(c)d(x_1,x_2,\omega)\right)^{1-\alpha} \\
&\quad\times\left(\sum_{c=1}^{K_C}\int_{[0,\infty)^3}\left|f_1(x_1,x_2,\omega|c)^{\alpha}-f_2(x_1,x_2,\omega|c)^{\alpha}\right|^{\frac{1}{\alpha}}p_C(c)d(x_1,x_2,\omega)\right)^{\alpha}.
\end{align*}
By using the definition of $C_{\alpha}(f_1,f_2)$, we see that \eqref{eq:hellinger_bound} follows from the above if we can prove that
\begin{align}
&\sum_{c=1}^{K_C}\int_{[0,\infty)^3}\left|f_1(x_1,x_2,\omega|c)^{\alpha}-f_2(x_1,x_2,\omega|c)^{\alpha}\right|^{\frac{1}{\alpha}}p_C(c)d(x_1,x_2,\omega) \nonumber \\
\leq&2\rho_H(\phi_{G,1},\phi_{G,2})^2+8\rho_H(\phi_{I,1},\phi_{I,2})^2. \label{eq:istep}
\end{align}
We begin by applying the reverse triangle inequality for the $L^{1/\alpha}$ norm repeatedly. Define to this end
\begin{align*}
B_1(t;x_1,x_2,c,\omega,y):=&\phi_{G,1}(y)h(\omega-t|c)\phi_{I,1}(x_1-t)\phi_{I,1}(x_2-t-y), \\
B_2(t;x_1,x_2,c,\omega,y):=&\phi_{G,2}(y)h(\omega-t|c)\phi_{I,2}(x_1-t)\phi_{I,2}(x_2-t-y), \\
A_1(y;x_1,x_2,c,\omega):=&\int_0^{\min(x_1,\omega)}B_1(t;x_1,x_2,c,y)dt, \\
A_2(y;x_1,x_2,c,\omega):=&\int_0^{\min(x_1,\omega)}B_2(t;x_1,x_2,c,y)dt.
\end{align*}
We have now by the reverse triangle inequality for $L^{1/\alpha}$ that for any $x_1,x_2,\omega\geq0$
\begin{align*}
&\Bigg|\left(\int_0^{x_2}A_1(y;x_1,x_2,c,\omega)dy\right)^{\alpha}-\left(\int_0^{x_2}A_2(y;x_1,x_2,c,\omega)dy\right)^{\alpha}\Bigg|^{\frac{1}{\alpha}} \\
=&\left|\left\|A_1(\cdot;x_1,x_2,c,\omega)^{\alpha}\right\|_{L^{1/\alpha}([0,x_2])}-\left\|A_2(\cdot;x_1,x_2,c,\omega)^{\alpha}\right\|_{L^{1/\alpha}([0,x_2])}\right|^{\frac{1}{\alpha}} \\
\leq&\left\|A_1(\cdot;x_1,x_2,c,\omega)^{\alpha}-A_2(\cdot;x_1,x_2,c,\omega)^{\alpha}\right\|_{L^{1/\alpha}([0,x_2])}^{\frac{1}{\alpha}} \\
=&\int_0^{x_2}\left|A_1(y;x_1,x_2,c,\omega)^{\alpha}-A_2(y;x_1,x_2,c,\omega)^{\alpha}\right|^{\frac{1}{\alpha}}dy \\
=&\int_0^{x_2}\left|\left(\int_0^{\min(x_1,\omega)}B_1(t;x_1,x_2,c,\omega,y)dt\right)^{\alpha}-\left(\int_0^{\min(x_1,\omega)}B_2(t;x_1,x_2,c,\omega,y)dt\right)^{\alpha}\right|^{\frac{1}{\alpha}}dy \\
=&\int_0^{x_2}\Big|\left\|B_1(\cdot;x_1,x_2,c,\omega,y)^{\alpha}\right\|_{L^{1/\alpha}([0,\min(x_1,\omega)])} \\
&\quad\quad\quad\quad-\left\|B_2(\cdot;x_1,x_2,c,\omega,y)^{\alpha}\right\|_{L^{1/\alpha}([0,\min(x_1,\omega)])}\Big|^{\frac{1}{\alpha}}dy \\
\leq&\int_0^{x_2}\left\|B_1(\cdot;x_1,x_2,c,\omega,y)^{\alpha}-B_2(\cdot;x_1,x_2,c,\omega,y)^{\alpha}\right\|_{L^{1/\alpha}([0,\min(x_1,\omega)])}^{\frac{1}{\alpha}}dy \\
=&\int_0^{x_2}\int_0^{\min(x_1,\omega)}\Bigg|\left(\phi_{G,1}(y)h(\omega-t|c)\phi_{I,1}(x_1-t)\phi_{I,1}(x_2-t-y)\right)^{\alpha} \\
&\quad\quad\quad\quad-\left(\phi_{G,2}(y)h(\omega-t|c)\phi_{I,2}(x_1-t)\phi_{I,2}(x_2-t-y)\right)^{\alpha}\Bigg|^{\frac{1}{\alpha}}dtdy.
\end{align*}
By using the above inequality chain, we obtain
\begin{align}
&\sum_{c=1}^{K_C}\int_{[0,\infty)^3}\left|f_1(x_1,x_2,\omega|c)^{\alpha}-f_2(x_1,x_2,\omega|c)^{\alpha}\right|^{\frac{1}{\alpha}}p_C(c)d(x_1,x_2,\omega) \nonumber \\
=&\sum_{c=1}^{K_C}\int_{[0,\infty)^3}\phi_W(\omega)n(\omega,c)p_C(c) \nonumber \\
&\times\Bigg|\left(\int_0^{x_2}\phi_{G,1}(y)\int_0^{\min(x_1,\omega)}h(\omega-t|c)\phi_{I,1}(x_1-t)\phi_{I,1}(x_2-t-y)dtdy\right)^{\alpha} \nonumber \\
&-\left(\int_0^{x_2}\phi_{G,2}(y)\int_0^{\min(x_1,\omega)}h(\omega-t|c)\phi_{I,2}(x_1-t)\phi_{I,2}(x_2-t-y)dtdy\right)^{\alpha}\Bigg|^{\frac{1}{\alpha}}d(x_1,x_2,\omega) \nonumber \\
=&\sum_{c=1}^{K_C}\int_{[0,\infty)^3}\phi_W(\omega)n(\omega,c)p_C(c) \nonumber \\
&\quad\times\Bigg|\left(\int_0^{x_2}A_1(y;x_1,x_2,c,\omega)dy\right)^{\alpha}-\left(\int_0^{x_2}A_2(y;x_1,x_2,c,\omega)dy\right)^{\alpha}\Bigg|^{\frac{1}{\alpha}}d(x_1,x_2,\omega) \nonumber \\
\leq&\sum_{c=1}^{K_C}\int_{[0,\infty)^3}\phi_W(\omega)n(\omega,c)p_C(c) \nonumber \\
&\quad\quad\times\int_0^{x_2}\int_0^{\min(x_1,\omega)}\Big|\left(\phi_{G,1}(y)h(\omega-t|c)\phi_{I,1}(x_1-t)\phi_{I,1}(x_2-t-y)\right)^{\alpha} \nonumber \\
&\quad\quad\quad-\left(\phi_{G,2}(y)h(\omega-t|c)\phi_{I,2}(x_1-t)\phi_{I,2}(x_2-t-y)\right)^{\alpha}\Big|^{\frac{1}{\alpha}}dtdyd(x_1,x_2,\omega). \label{eq:step1}
\end{align}
Above we have an iterated integral over a non-negative function, we may thus re-arrange the order of integration and use substitution. We substitute below $a=x_1-t$ for $x_1$ and $b=x_2-t-y$ for $x_2$. Note that we implicitly take care of the integration bounds by using the indicator function and the fact that all densities are zero on the negative real line. Hence, we can continue the above inequality chain
\begin{align}
\eqref{eq:step1}=&\sum_{c=1}^{K_C}\int_{[0,\infty)^5}\Ind(y\leq x_2, t\leq\min(x_1,\omega))\phi_W(\omega)n(\omega,c)p_C(c)h(\omega-t|c) \nonumber \\
&\quad\quad\times\Big|\left(\phi_{G,1}(y)\phi_{I,1}(x_1-t)\phi_{I,1}(x_2-t-y)\right)^{\alpha} \nonumber \\
&\quad\quad\quad-\left(\phi_{G,2}(y)\phi_{I,2}(x_1-t)\phi_{I,2}(x_2-t-y)\right)^{\alpha}\Big|^{\frac{1}{\alpha}}dx_1\,dx_2\,dt\,dy\,d\omega \nonumber \\
=&\sum_{c=1}^{K_C}\int_{[0,\infty)^5}\Ind(y\leq x_2, t\leq\min(a+t,\omega))\phi_W(\omega)n(\omega,c)p_C(c)h(\omega-t|c) \nonumber \\
&\quad\quad\times\Big|\left(\phi_{G,1}(y)\phi_{I,1}(a)\phi_{I,1}(x_2-t-y)\right)^{\alpha} \nonumber \\
&\quad\quad\quad-\left(\phi_{G,2}(y)\phi_{I,2}(a)\phi_{I,2}(x_2-t-y)\right)^{\alpha}\Big|^{\frac{1}{\alpha}}da\,dx_2\,dt\,dy\,d\omega \nonumber \\
=&\sum_{c=1}^{K_C}\int_{[0,\infty)^5}\Ind(0\leq b+t, t\leq\min(a+t,\omega))\phi_W(\omega)n(\omega,c)p_C(c)h(\omega-t|c) \nonumber \\
&\quad\quad\times\Big|\left(\phi_{G,1}(y)\phi_{I,1}(a)\phi_{I,1}(b)\right)^{\alpha}-\left(\phi_{G,2}(y)\phi_{I,2}(a)\phi_{I,2}(b)\right)^{\alpha}\Big|^{\frac{1}{\alpha}}da\,db\,dt\,dy\,d\omega. \label{eq:step2}
\end{align}
Note next that the indicator equals actually $\Ind(t\in[0,\omega])$. We again interchange the order of integration, to integrate with respect to $t$ first and then with respect to $\omega$. By doing this, recalling the form of $\phi_{T_1}$ in Assumption (M), keeping in mind that $|x^{2\alpha}-y^{2\alpha}|\leq|x-y|^{2\alpha}$ for all  $x,y\geq0$ (since $2\alpha\in(0,1)$) and that $(x-y)^2\leq2x^2+2y^2$ for all $x,y$, we continue
\begin{align}
\eqref{eq:step2}=&\sum_{c=1}^{K_C}\int_{[0,\infty)^5}\phi_W(\omega)p_C(c)\phi_{T_1}(t|\omega,c)dt\,d\omega \nonumber \\
&\quad\quad\times\Big|\left(\phi_{G,1}(y)\phi_{I,1}(a)\phi_{I,1}(b)\right)^{\alpha}-\left(\phi_{G,2}(y)\phi_{I,2}(a)\phi_{I,2}(b)\right)^{\alpha}\Big|^{\frac{1}{\alpha}}da\,db\,dy \nonumber \\
=&\int_{[0,\infty)^3}\Big|\left(\phi_{G,1}(y)\phi_{I,1}(a)\phi_{I,1}(b)\right)^{\alpha}-\left(\phi_{G,2}(y)\phi_{I,2}(a)\phi_{I,2}(b)\right)^{\alpha}\Big|^{\frac{1}{\alpha}}da\,db\,dy \label{eq:ibound} \\
\leq&\int_{[0,\infty)^3}\Big|\left(\phi_{G,1}(y)\phi_{I,1}(a)\phi_{I,1}(b)\right)^{\frac{1}{2}}-\left(\phi_{G,2}(y)\phi_{I,2}(a)\phi_{I,2}(b)\right)^{\frac{1}{2}}\Big|^2da\,db\,dy \nonumber  \\
\leq&\int_{[0,\infty)^3}2\left(\phi_{G,1}(y)^{\frac{1}{2}}-\phi_{G,2}(y)^{\frac{1}{2}}\right)^2\phi_{I,1}(a)\phi_{I,1}(b)da\,db\,dy \nonumber  \\
&\quad\quad+2\int_{[0,\infty)^3}\phi_{G,2}(y)\left(\phi_{I,2}(a)^{\frac{1}{2}}\phi_{I,2}(b)^{\frac{1}{2}}-\phi_{I,1}(a)^{\frac{1}{2}}\Phi_{I,1}(b)^{\frac{1}{2}}\right)^2da\,db\,dy  \nonumber  \\
=&2\rho_H(\phi_{G,1},\phi_{G,2})^2+2\int_{[0,\infty)^2}\left(\phi_{I,2}(a)^{\frac{1}{2}}\phi_{I,2}(b)^{\frac{1}{2}}-\phi_{I,1}(a)^{\frac{1}{2}}\phi_{I,1}(b)^{\frac{1}{2}}\right)^2da\,db \nonumber  \\
\leq&2\rho_H(\phi_{G,1},\phi_{G,2})^2+4\int_{[0,\infty)^2}\left(\phi_{I,2}(a)^{\frac{1}{2}}-\phi_{I,1}(a)^{\frac{1}{2}}\right)^2\phi_{I,2}(b)d\,db \nonumber  \\
&\quad\quad+4\int_{[0,\infty)^2}\phi_{I,1}(a)\left(\phi_{I,1}(b)^{\frac{1}{2}}-\phi_{I,2}(b)^{\frac{1}{2}}\right)^2da\,db \nonumber  \\
=&2\rho_H(\phi_{G,1},\phi_{G,2})^2+8\rho_H(\phi_{I,1},\phi_{I,2})^2. \nonumber 
\end{align}
This is \eqref{eq:istep} and the proof is complete.
\end{proof}

\begin{proof}[Proof of Lemma \ref{lem:bracketing}]
Denote for any $m\in\IN_0$, $\Theta_m:=\{\theta\in\IR^{m+1}:\,\|\theta\|_2=1\}$. The proof of this Lemma uses the following strategy which is similar to Lemma 2.1 in \citet{O87}. In the interest of completeness we give the detailed proof: Let $\delta\in(0,\sqrt{3/2}/2]$ be given. We define for any $\theta_1\in\Theta_{m_1}$ and $\theta_2\in\Theta_{m_2}$ the $\delta$-ball
$$B_{\delta}(\theta_1,\theta_2):=\left\{\left(\tilde{\theta}_1,\tilde{\theta}_2\right)\in\Theta_{m_1}\times\Theta_{m_2}:\,\|\theta_1-\tilde{\theta}_1\|_2\leq\delta, \|\theta_2-\tilde{\theta}_2\|_2\leq\delta\right\}.$$
Find now a set $\left((\theta_{1,i},\theta_{2,i})\right)_{i=1,...,N(\delta)}\subseteq\Theta_{m_1}\times\Theta_{m_2}$ such that
$$\bigcup_{i=1}^{N(\delta)}B_{\delta}(\theta_{1,i},\theta_{2,i})\supseteq \Theta_{m_1}\times\Theta_{m_2}.$$
In order to bound $N(\delta)$ we construct a specific collection of pairs: Consider a grid of $[0,\pi]^{m_1}$ with side length $\alpha=2\delta\sqrt{2/3}$ and construct $\theta_{1,i}$ by taking the grid points as polar coordinates (with radius $1$). Then, it is clear that for any $\theta\in\Theta_{m_1}$, there is a grid point $\theta_{1,i}$ such that the difference between any two angles of the polar representations of $\theta$ and $\theta_{1,i}$ is smaller than $\alpha/2=\delta\sqrt{2/3}$. By Lemma \ref{lem:polar} below and symmetry of the polar coordinates, we find that $\|\theta-\theta_{1,i}\|_2\leq \sqrt{3/2}\alpha/2=\delta$. The size of this grid can be bounded by $(\pi/\alpha+1)^{m_1}$. We repeat this construction for $\Theta_{m_2}$ and obtain
\begin{equation}
\label{eq:covering_number}
N(\delta)\leq\left(\frac{\sqrt{3}\pi}{2\sqrt{2}\delta}+1\right)^{m_1+m_2}.
\end{equation}
The brackets are now defined as
\begin{align*}
l_i(x_1,x_2,\omega,c):=\inf_{(\theta_1,\theta_2)\in B_{\delta}(\theta_{1,i},\theta_{2,i})}f_{\phi_{I,\theta_1},\phi_{G,\theta_2}}(x_1,x_2,\omega|c)p_C(c), \\
u_i(x_1,x_2,\omega,c):=\sup_{(\theta_1,\theta_2)\in B_{\delta}(\theta_{1,i},\theta_{2,i})}f_{\phi_{I,\theta_1},\phi_{G,\theta_2}}(x_1,x_2,\omega|c)p_C(c).
\end{align*}
For any $f_{\phi_{I,\theta_1},\phi_{G,\theta_2}}p_C\in\mathcal{F}_{m_1,m_2}$ we find thus first a pair $(\theta_{1,i},\theta_{2,i})$ such that $(\theta_1,\theta_2)\in B_{\delta}(\theta_{1,i},\theta_{2,i})$ and thus also $l_i\leq f_{\phi_{I,\theta_1},\phi_{G,\theta_2}}p_C\leq u_i$. It remains to compute $\rho_H(l_i,u_i)$. To this end, we firstly see that the same arguments which lead to \eqref{eq:ibound} (for $\alpha=1/2$) give us here the following (the $\sup$ refers always to the supremum over all pairs $(\theta_1,\theta_2),(\tilde{\theta}_1,\tilde{\theta}_2)\in\Theta_{m_1}\times\Theta_{m_2}$ such that $\|\theta_1-\tilde{\theta}_1\|_2\leq\delta$ and $\|\theta_2-\tilde{\theta}_2\|_2\leq\delta$)
\begin{align*}
&\rho_H(l_i,u_i)^2 \\
\leq&\sum_{c=1}^{K_C}\int_{[0,\infty)^3}\sup\left(f_{\phi_{I,\theta_1},\phi_{G,\theta_2}}(x_1,x_2,\omega)^{\frac{1}{2}}-f_{\phi_{I,\tilde{\theta}_1},\phi_{G,\tilde{\theta}_2}}(x_1,x_2,\omega)^{\frac{1}{2}}\right)^2p_C(c)d(x_1,x_2,\omega) \\
\leq&\int_{[0,\infty)^3}\sup\left(\left(\phi_{G,\theta_1}(y)\phi_{I,\theta_2}(a)\phi_{I,\theta_2}(b)\right)^{\frac{1}{2}}-\left(\phi_{G,\tilde{\theta}_1}(y)\phi_{I,\tilde{\theta}_2}(a)\phi_{I,\tilde{\theta}_2}(b)\right)^{\frac{1}{2}}\right)^2dydadb \\
\leq&\int_{[0,\infty)^3}e^{-y-a-b}\sup\Bigg(\sum_{k=0}^{m_1}\left(\theta_{1,k}-\tilde{\theta}_{1,k}\right)L_k(y)\sum_{k=0}^{m_2}\theta_{2,k}L_k(a)\sum_{k=0}^{m_2}\theta_{2,k}L_k(b) \\
&-\sum_{k=0}^{m_1}\tilde{\theta}_{1,k}L_k(y)\Bigg(\sum_{k=0}^{m_2}\left(\tilde{\theta}_{2,k}-\theta_{2,k}\right)L_k(a)\sum_{k=0}^{m_2}\tilde{\theta}_{2,k}L_k(b) \\
&\quad\quad\quad\quad\quad\quad-\sum_{k=0}^{m_2}\theta_{2,k}L_k(a)\sum_{k=0}^{m_2}\left(\theta_{2,k}-\tilde{\theta}_{2,k}\right)L_k(b)\Bigg)\Bigg)^2 dydadb\\
\leq&\int_{[0,\infty)^3}e^{-y-a-b}\Bigg(2\delta^2\sum_{k=0}^{m_1}L_k(y)^2\sum_{k=0}^{m_2}L_k(a)^2\sum_{k=0}^{m_2}L_k(b)^2 \\
&+2\sum_{k=0}^{m_1}L_k(y)^2\left(2\delta^2\sum_{k=0}^{m_2}L_k(a)^2\sum_{k=0}^{m_2}L_k(b)^2+2\delta^2\sum_{k=0}^{m_2}L_k(a)^2\sum_{k=0}^{m_2}L_k(b)^2\right)\Bigg)dydadb \\
=&10\delta^2m_1m_2^2,
\end{align*}
where we used the Cauchy-Schwarz-Inequality in between and the integral properties of the Laguerre polynomials at the end. Thus, when putting $\delta
=\epsilon\left(10m_1m_2^2\right)^{-1/2}$ (the condition on $\delta$ is fulfilled by the assumption on $\epsilon$) we find together with \eqref{eq:covering_number}
$$\mathcal{N}_{[]}(\epsilon,\mathcal{F}_{m_1,m_2},\rho_H)\leq\left(\frac{\pi\sqrt{15m_1m_2^2}}{2\epsilon}+1\right)^{m_1+m_2}\leq\left(\frac{\pi\sqrt{15m_1m_2^2}}{\epsilon}\right)^{m_1+m_2}$$
and the proof is complete.
\end{proof}

\begin{lemma}
\label{lem:polar}
Let $n\geq2$ and $\delta\in[0,1/2]$. Denote by $e_1:=(1,0,...,0)'$ the first unit vector of $\IR^n$. Let $x\in\IR^n$ have $\|x\|_2=1$ and angles of polar coordinates $\psi_1,...,\psi_{n-1}\in[0,\delta]$. Then, $\|e_1-x\|_2\leq\delta\sqrt{3/2}$.
\end{lemma}
\begin{proof}
By a Taylor expansion, we have for any $\psi\in[0,2\pi]$ (below $\psi^*$ denotes different intermediate values between $0$ and $\psi$)
\begin{align*}
|\sin(\psi)|=&|\cos(\psi^*)|\cdot|\psi|\leq|\psi| \\
|\cos(\psi)-1|=&|\sin(\psi^*)|\cdot|\psi|\leq|\psi|^2.
\end{align*}
By the definition of polar coordinates, we compute (for $n=2$ the sum disappears)
\begin{align*}
&\|e_1-x\|_2^2=(1-\cos(\psi_1))^2+\prod_{k=1}^{n-1}\sin(\psi_k)^2+\sum_{i=2}^{n-1}\prod_{k=1}^{i-1}\sin(\psi_k)^2\cos(\psi_i)^2 \\
\leq&\psi_1^4+\prod_{k=1}^{n-1}\psi_k^2+\sum_{i=2}^{n-1}\prod_{k=1}^{i-1}\psi_k^2 \leq\delta^4+\delta^{2(n-1)}+\sum_{i=2}^{n-1}\delta^{2(i-1)}\leq\delta^2\frac{\delta^2-\delta^4+1}{1-\delta^2}.
\end{align*}
The statement follows since $(\delta^2-\delta^4+1)/(1-\delta^2)\leq3/2$ for $\delta\in[0,1/2]$.
\end{proof}

\end{document}